\documentclass[11pt]{article}

\usepackage{graphicx}
\usepackage{bm}
\usepackage{comment}
\usepackage{psfrag}
\usepackage{latexsym,amsmath,amsfonts,amscd, amsthm}
\usepackage{changebar}
\usepackage{color}
\usepackage{bm}
\usepackage{tikz}
\usepackage{multirow}
\usepackage{subfigure}
\usepackage{xcolor}
\usepackage{soul} % use this (many fancier options) (to strike off), use \st

\usetikzlibrary{arrows,backgrounds,snakes}
\usepackage{extarrows}

\usepackage{chngcntr}

\counterwithin*{equation}{section}

\newtheoremstyle{plainNoItalics}{}{}{\normalfont}{}{\bfseries}{.}{ }{}

\theoremstyle{plain}
\newtheorem{thm}{Theorem}[section]

\theoremstyle{plainNoItalics}

\newtheorem{lem}[thm]{Lemma}
\newtheorem{defn}[thm]{Definition}
\newtheorem{rem}[thm]{Remark}
\newtheorem{prop}[thm]{Proposition}

\newcommand{\beq}{\begin{equation}}
\newcommand{\eeq}{\end{equation}}
\newcommand{\beqa}{\begin{eqnarray}}
\newcommand{\eeqa}{\end{eqnarray}}
\newcommand{\bit}{\begin{itemize}}
\newcommand{\eit}{\end{itemize}}
\newcommand{\bedef}{\begin{defn}}
\newcommand{\edefn}{\end{defn}}
\newcommand{\bpro}{\begin{prop}}
\newcommand{\epro}{\end{prop}}

\newcommand{\df}{\partial}

\newcommand{\Dx}{\Delta x}
\newcommand{\Dt}{\Delta t}

\newcommand{\mP}{\mathcal P}

\newcommand{\eps}{\varepsilon}
\newcommand{\mD}{{\mathcal D}}
\newcommand{\Ox}{{\Omega_x}}
\newcommand{\Ov}{{\Omega_v}}
\newcommand{\xL}{{x_{i-\frac{1}{2}}}}
\newcommand{\xR}{{x_{i+\frac{1}{2}}}}
\newcommand{\iL}{{i-\frac{1}{2}}}
\newcommand{\iR}{{i+\frac{1}{2}}}

\newcommand{\testR}{{\phi}}   % test function for rho
\newcommand{\testQ}{{\varphi}}   % test function for q
\newcommand{\testG}{{\psi}}   % test function for g
\newcommand{\testU}{{\eta}}  % test function for u, which is the L^2 projection of q in the polynomial space

\newcommand{\bI}{{\bf{I}}}
\newcommand{\mK}{{\mathcal K}}

\newcommand{\bh}{{b_{h,v}}}

\newcommand{\lgl}{\langle}
\newcommand{\rgl}{\rangle}

\newcommand{\aaa}{\widehat{C}_{inv}||v^2||_\infty}
\newcommand{\bbb}{C_{inv}||v||_\infty}

\newcommand{\AP}{ {\eps,\Dt,h} }

\newcommand{\LIM}{ {\Dt,h} }

\newcommand{\tauzero}{ \widehat{\tau}_{\eps,h,0}  }
\newcommand{\tauone}{ \widehat{\tau}_{\eps,h,1}  }
\newcommand{\tautwo}{ \widehat{\tau}_{\eps,h,2}  }

\newcommand{\lamzero}{ \lambda_{0}}
\newcommand{\lamone}{ \lambda_{1} }
\newcommand{\lamtwo}{ \lambda_{2} }
\newcommand{\dtstabzero}{ \Dt_{\textrm{stab}, 0} }
\newcommand{\dtstabk}{ \Dt_{\textrm{stab}} }

\begin{document}
\title{Stability-enhanced AP IMEX1-LDG method: energy-based stability and rigorous AP property}
\author{Zhichao Peng\thanks{Department of Mathematical Sciences, Rensselaer Polytechnic Institute, Troy, NY 12180, U.S.A. {\tt pengz2@rpi.edu}}
   \and
   Yingda Cheng\thanks{Department of Mathematics, Department of  Computational Mathematics, Science and Engineering, Michigan State University,
               East Lansing, MI 48824 U.S.A.
          {\tt ycheng@msu.edu}. Research is supported by NSF grants  DMS-1453661 and DMS-1720023.}
   \and
Jing-Mei Qiu\thanks{Department of Mathematical Sciences, University of Delaware, Newark, DE, 19716, U.S.A.
 {\tt jingqiu@udel.edu}. 
 Research is supported by NSF grant DMS-1818924 and AFOSR grant 
 FA9550-18-1-0257.} 
   \and
 Fengyan Li\thanks{Department of Mathematical Sciences, Rensselaer Polytechnic Institute, Troy, NY 12180, U.S.A.
 {\tt lif@rpi.edu}. Research is supported by NSF grants DMS-1719942 and DMS-1913072.}}
\maketitle
\abstract{In our recent work \cite{peng2018stability}, a family of high order asymptotic preserving (AP) methods, termed as IMEX-LDG methods, are designed to solve some linear kinetic transport equations,  including the one-group transport equation in slab geometry and the telegraph equation,  in a diffusive scaling. As the Knudsen number $\varepsilon$ goes to zero, the limiting schemes are implicit discretizations to the limiting diffusive equation.  Both  Fourier analysis and numerical experiments imply the methods are unconditionally stable in the diffusive regime when  $\varepsilon\ll1$. 
In this paper, we develop an energy approach to establish the numerical stability of 
the IMEX1-LDG method, the sub-family of the methods that is first order accurate  in time and arbitrary order in space, for the model with general material properties.
Our analysis is the  first to simultaneously confirm  unconditional stability when $\varepsilon\ll 1$ and the uniform stability property with respect to $\varepsilon$.
 To capture the unconditional stability,  a novel discrete energy is introduced by better exploring the contribution of the scattering term in different regimes.  
A general form of the weight function, introduced to obtain the unconditional stability for $\varepsilon\ll 1$, is also for the first time considered in such stability analysis.
Based on the uniform stability,  a rigorous asymptotic analysis  is then carried out to show the AP property.  
}

\section{Introduction}
\label{sec0:intro}

In this paper, we continue our efforts in devising and advancing mathematical understanding of asymptotic preserving (AP) methods to solve time-dependent multi-scale kinetic transport equations within  the discontinuous Galerkin (DG) framework  \cite{AP1, jang2014analysis, peng2018stability}. Particularly,  we focus on establishing energy-type numerical stability and the  AP property for some methods  proposed  in  \cite{peng2018stability} for the following model equation,
 \begin{equation}
	\label{eq:mmd:LJ0}
	{\mP}^\eps: \qquad 
	\eps f_t + v \partial_x f = \frac{\sigma_s}{\eps}\left(\langle f\rangle - f\right)-\eps \sigma_a f,
\end{equation}
with periodic boundary conditions.  The function $f=f(x,v,t)$ is  the probability distribution  
function of the particles, with the space variable $x\in \Omega_x\subset {\mathbb R}$, velocity variable $v\in \Omega_v\subset {\mathbb R}$, and time  $t\geq 0$. 
 $\sigma_s(x)>0 $ and $\sigma_a(x)\geq 0$ are the scattering and absorption coefficients, respectively.
${\mathcal L}(f)=\lgl f\rgl-f$ defines a  scattering operator, where  $\langle f\rangle: = \int_{\Omega_v} f d \nu$  and  $\nu$ is a measure of the velocity space satisfying $\int_\Ov 1 d\nu=1$.   The parameter $\eps>0$ is the dimensionless Knudsen number, defined as the ratio of the mean free path of the particles over the characteristic length of the system.  The model \eqref {eq:mmd:LJ0} is in a diffusive scaling, and as $\eps\rightarrow 0$, it approaches its  diffusive limit
\begin{equation}
	\label{eq;mmd:LJ0_lim}
	{\mP}^0: \qquad 
\partial_t\rho=\langle v^2\rangle \partial_x (\partial_x \rho/\sigma_s)-\sigma_a \rho.
\end{equation}
Here $\rho= \lgl f \rgl$ is the  macroscopic density. Though seemingly simple, the equation in \eqref{eq:mmd:LJ0} provides
a prototype model to study many realistic problems such as in neutron transport or radiative transfer theory both numerically and mathematically.

To  simulate multi-scale models like that in   \eqref{eq:mmd:LJ0} effectively and reliably for a broad range of value for $\eps$, AP methods are widely recognized by the scientific community (see e.g. review papers  \cite{jin2010asymptotic, degond2011asymptotic}).  These methods  are designed for the  governing model with $\eps>0$. 
Additionally when $\eps\rightarrow0$, the methods become consistent  and stable discretizations for the limiting model as in \eqref{eq;mmd:LJ0_lim} even on under-resolved meshes with $\Dx, \Dt\gg \eps$.  Hence, AP methods provide  a natural transition of different regimes in multi-scale simulations.  AP methods usually involve some level of implicit treatment to deal with the stiffness of the model when $\eps\ll 1$.
%,  while balancing  to have reasonable computational cost. 
It is known that stability alone does not guarantee the scheme to capture the correct asymptotic limit \cite{caflisch1997uniformly, naldi1998numerical}.

 In our recent work \cite{peng2018stability}, a family of high order  AP methods,  termed as IMEX-LDG methods, are designed for \eqref{eq:mmd:LJ0}.   The methods are based on the reformulation of the equation, and involve local DG (LDG) discretization in space \cite{cockburn1998local}, globally stiffly accurate implicit-explicit (IMEX) Runge-Kutta (RK) methods in time \cite{boscarino2013imex}, and a judicially  chosen IMEX strategy. The reformulation has  two steps: micro-macro decomposition \cite{liu2004boltzmann,Lemou_No_BC}, and addition/subtraction of a $\omega$-weighted diffusive term \cite{boscarino2013imex}. The latter is introduced to obtain fully implicit limiting schemes as $\eps\rightarrow 0$, to achieve  unconditional stability of the methods in the diffusive regime with $\eps\ll 1$,  hence to  circumvent the otherwise stringent parabolic type time step condition in this regime, namely, $\Dt=O(\Dx^2)$, of  many AP schemes whose limiting schemes are explicit
 \cite{jin1998diffusive, klar1998asymptotic, Lemou_No_BC,  AP1}. Using globally stiffly accurate IMEX RK methods in time, 
 and LDG methods in space with suitable numerical fluxes,
  the IMEX-LDG methods project the numerical solutions to the local equilibrium at both inner stages and full RK steps in the limit of  $\eps\rightarrow 0$, and this is important for the AP property and seemingly also for  accuracy (see appendix of \cite{peng2018stability}).  
 In \cite{peng2018stability},  unconditional stability in the diffusive regime is observed numerically, and is confirmed by a Fourier-type stability analysis applied to the  two-velocity telegraph equation with $\Ov=\{-1, +1\}$, and constant material properties $\sigma_s=1$, $\sigma_a=0$. 
 
In this work we restrict our attention to the IMEX1-LDG method, the sub-family of the methods in  \cite{peng2018stability} that  is  first order accurate  in time and arbitrary order in space, and examine it systematically for the model with the general 
 material properties, namely with the spatially varying scattering  and absorption coefficients  $\sigma_s(x)$ and $\sigma_a(x)$. Our main objectives are two-fold.  
The first is to establish unconditional stability in the diffusive regime with $\eps\ll1$ {\em as well as}  uniform stability with respect to $\eps$.  By following an energy approach as in \cite{liu2010analysis, jang2014analysis}, one can get uniform stability yet fails to capture the unconditional stability for $\eps\ll 1$. Note the methods examined in \cite{liu2010analysis, jang2014analysis} in the limit of $\eps\rightarrow 0$ are explicit. We instead propose and work with a new  notion of $\mu$-stability, and get the stability we want by better exploring the contribution of the scattering operator. 
The stability results up to this point depend on  a parameter $\mu$. An intricate algebraic-based optimization  with respect to the admissible $\mu$ is  subsequently followed, to further maximize the unconditional stability region, while  also maximizing the allowable time step size in the regime when the method is  conditionally stable.  
 As our second objective, a rigorous asymptotic analysis is proved to show the AP property based on the uniform stability. 
 To our best knowledge, our analysis is the first to capture unconditional stability when $\eps\ll 1$ along with uniform stability property for the model \eqref{eq:mmd:LJ0} with general material properties. A general form of  the weight function $\omega$ is also for the first time considered in such stability analysis. 
In this work, we keep the velocity variable continuous, and our analysis can be easily adapted  when the velocity variable is further discretized such as by discrete ordinates or $P_N$ methods \cite{pomraning1973equations}.  Our analysis can also be extended to AP methods with the same IMEX strategy yet with other spatial discretizations, as long as they satisfy some key properties, such as the adjoint property in \eqref{eq:lem1:1} (also see Lemma 3.5 
in \cite{peng2018stability})
and the stabilization as in \eqref{eq:upwind} due to the upwind treatment.   Though not presented here,   {\em a priori} error estimates can follow similarly as in \cite{jang2014analysis}, and they are uniform in $\eps$ for  smooth enough solutions with uniform bounds in $\eps$ under the relevant Sobolev norms. What seems to be more challenging and left to our  future endeavor
is to obtain  the stability analysis for IMEX-LDG methods with higher order temporal accuracy.

Finally we want to briefly review some related  literature especially in establishing numerical stability 
of AP methods for kinetic transport models in a diffusive scaling.  One commonly used approach is Fourier type analysis. 
For the telegraph equation with $\Ov=\{-1, +1\}$,  an analytical time step condition is given in \cite{Lemou_No_BC} via Fourier analysis to ensure uniform  $L^2$-stability of a first order finite difference AP method, while in \cite{peng2018stability}, necessary conditions on $\eps, \Dx, \Dt$  are obtained numerically for the $p$-th order IMEX-LDG AP scheme ($p=1, 2, 3$) to ensure an $L^2$ energy non-increasing in time. The results seem to be uniform in $\eps$, with  unconditional stability captured for $\eps\ll 1$.    
Klar and Unterreiter in \cite{klar2002uniform} considered a formally first-order in time and second-order in space  AP scheme for the one-group transport equation with $\Ov=[-1, 1]$ and established uniform stability by first establishing the result in Fourier space and then transforming  it back to the physical space.  Their analysis assumes the $H^1$ smoothness of the initial data.  It is known that Fourier-type  analysis requires uniform meshes and the models being linear and constant-coefficient. Energy-based stability analysis on the other hand does not pose these restrictions, yet they are not always easy to get. In \cite{liu2010analysis}, Liu and Mieussens revisited the first order AP method in \cite{Lemou_No_BC} for a more general kinetic transport model and proved uniform stability following an energy approach. A similar analysis
is carried out in \cite{jang2014analysis} for the first order in time DG-IMEX1 method in \cite{AP1}. Based on the uniform stability analysis, error estimates and rigorous asymptotic analysis are also established in \cite{jang2014analysis}.  
In both \cite{peng2018stability} and here in this work, we want to capture the unconditional stability in the diffusive regime in addition to the uniform stability. Few other  theoretical works, among many, for AP methods include uniform consistency \cite{caflisch1997uniformly, klar1998asymptotic}, uniform convergence \cite{golse1999convergence,filbet2013analysis} based on the commuting diagram of AP schemes (see Fig 1.1 in \cite{golse1999convergence}), and a recent work on uniform accuracy with   IMEX multi-step methods \cite{hu2019uniform}.

The remaining of the paper is organized as follows.  In Section \ref{sec1:scheme}, we review and extend  the IMEX1-LDG method in \cite{peng2018stability}  to our model \eqref{eq:mmd:LJ0} with general material properties.  
Section \ref{sec3:stab} presents  main results on numerical stability.   Here several theorems, including Theorem \ref{thm:stab_nmu} and Theorem \ref{thm:mu}, are stated to obtain uniform stability, while capturing the unconditional stability in the diffusive regime.  
An optimization step is carried out in Theorem \ref{thm:stab} to find the best value of the parameter $\mu$ in the notion of $\mu$-stability in order to optimize  the stability results. Once uniform stability is available, the AP property of the method is stated in Theorem \ref{thm:AP} in Section \ref{sec:ap}. The proofs of all major theorems are presented in Sections \ref{sec:stab:proofs}-\ref{sec:ap:proofs} for better readability.
%, followed by concluding remarks in Section \ref{sec:summary}.

\section{The IMEX1-LDG scheme }
\label{sec1:scheme}

In this section, we will review the IMEX1-LDG method proposed in \cite{peng2018stability} and extend it more systematically to the model \eqref{eq:mmd:LJ0} with general material properties $\sigma_s(x)$ and $\sigma_a(x)$, both  being in $L^\infty(\Omega_x)$ and satisfying $\sigma_M\geq \sigma_s(x)\geq\sigma_m >0, \sigma_a(x)\geq 0, \forall x\in\Ox$. The boundary conditions in space are periodic, 
and the velocity variable $v$ will not be discretized.

Two examples of the model \eqref{eq:mmd:LJ0} will be examined.  One is the one-group transport equation in slab geometry. Here $\Omega_v=[-1,1]$ and the measure $\nu$ is defined as
$
\int_{\Omega_v} f d\nu=\frac{1}{2}\int_{\Omega_v} f(x,v,t) dv,
$
with $dv$ being the standard Lebesgue measure. The other is the telegraph equation  with $\Omega_v=\{-1,1\}$, and $\nu$ is  a discrete measure, given as
$
\int_{\Omega_v}f d\nu=\frac{1}{2}\left(f(x,v=1,t)+f(x,v=-1,t)\right).
$
 There is little difference in  the formulation and analysis of the IMEX1-LDG method for both examples.

\subsection{Reformulation}
\label{sec1.1:reformat}

The IMEX1-LDG method is defined based on a reformulation of \eqref{eq:mmd:LJ0}, which is obtained in several steps. As the first step, we rewrite the model into its  micro-macro decomposition  \cite{liu2004boltzmann, Lemou_No_BC}. Let $L^2(\Omega_v,\nu)$ be the square integrable space in $v$, with the inner product $\langle f , g  \rangle :=  \langle fg \rangle$.  Let $\Pi$ be the $L^2$ projection onto  $\textrm{Null}({\mathcal L})=\text{Span} \{1\}$, $\mathbf{I}$  be the identify operator, and  $\rho:=  \langle f \rangle=\Pi f$  be the macroscopic density. Then $f$ can be decomposed orthogonally into $f = \rho+\varepsilon g$, with $\rho$ and $g$ satisfying 
\begin{subequations}
\label{eq:mmd:LJ}
\begin{align}
&\partial_t\rho+\partial_x \langle vg \rangle  =-\sigma_a \rho,
\label{eq:mmd:LJ:1} \\
	&\partial_t g  +\frac{1}{\varepsilon}  (\mathbf{I}-\Pi)(v\partial_xg) +\frac{1}{\varepsilon^2 }v\partial_x\rho = -\frac{\sigma_s}{\varepsilon^2 }g-\sigma_a g. \label{eq:mmd:LJ:2}
\end{align}
\end{subequations}
This is the micro-macro decomposition.
As $\eps\rightarrow 0$, the equations \eqref{eq:mmd:LJ} formally become
\begin{equation}
\label{eq:limit:diff:1stOrder}
\partial_t\rho+\partial_x \langle vg \rangle  =-\sigma_a\rho,\quad  \sigma_s g=- v\partial_x\rho,
\end{equation}
which is a first order form of the limiting diffusion equation,
\begin{equation}
\partial_t\rho=\langle v^2\rangle\partial_x\left(\partial_x\rho/\sigma_s \right)-\sigma_a\rho,
\label{eq:limit:diff}
\end{equation}
equipped with the compatible  initial condition.
The relation $\sigma_s g=-v\df_x\rho$ in \eqref{eq:limit:diff:1stOrder} will be referred to as   the {\em local equilibrium}.  For the telegraph equation, the diffusion constant is $\lgl v^2\rgl=1$, while for the one-group transport equation in slab geometry,  $\langle v^2\rangle=1/3$.

As the second step,  
a weighted diffusion term, $\omega \langle v^2\rangle  \partial_x ( \partial_x\rho/\sigma_s)$, is added  to both sides of \eqref{eq:mmd:LJ:1}, leading to 
\begin{subequations}
\label{eq:mmd:DIFF}
	\begin{align}
	&\partial_t\rho+\partial_x \langle vg \rangle+\omega \langle v^2\rangle \partial_x\left( \partial_x\rho/\sigma_s\right)  =\omega \langle v^2\rangle \partial_x\left(\partial_x\rho/\sigma_s\right)-\sigma_a\rho,\\
	&\partial_t g  +\frac{1}{\varepsilon}  (\mathbf{I}-\Pi)(v\partial_xg) +\frac{1}{\varepsilon^2 }v\partial_x\rho = -\frac{\sigma_s}{\varepsilon^2 }g-\sigma_a g.
	\end{align}
\end{subequations}
Here the weight function $\omega$ is non-negative and bounded. It is {\em independent} of $x$ and can depend on $\eps$, satisfying
\begin{equation}
\omega\rightarrow1, \quad \text{as} \quad \eps\rightarrow 0.
\label{w:prop:0}
\end{equation}
Additional properties desired for $\omega$ in general and considered  specifically in this work will be discussed in next subsection. 
The idea of reformulating a kinetic transport model in the diffusive scaling based on adding and subtracting a diffusive term was previously used in  \cite{boscarino2013imex} and \cite{dimarco2014implicit} to remove some parabolic stiffness in designing  AP schemes. One advancement we made in \cite{peng2018stability} and here is to improve the mathematical understanding of the  desired property and the role of the weight function  $\omega$, and such advancement can guide one to choose $\omega$  in practice.

With  the auxiliary variables $q=\partial_x \rho$ and $u=q/\sigma_s$, the  system \eqref{eq:mmd:DIFF} can also be written in its first order form
\begin{subequations}
\label{eq:mmd:Q}
	\begin{align}
	& q=\partial_x \rho, \qquad 	u=q/\sigma_s,  \label{eq:mmd:Q:0}\\
	& \partial_t\rho+\partial_x \langle v(g+\omega vu) \rgl=\omega \langle v^2\rangle \partial_x u -\sigma_a\rho,
 \label{eq:mmd:Q:1} \\
	&\partial_t g  +\frac{1}{\varepsilon}  (\mathbf{I}-\Pi)(v\partial_xg) +\frac{1}{\varepsilon^2 }v\partial_x\rho = -\frac{\sigma_s}{\varepsilon^2 }g-\sigma_a g,\label{eq:mmd:Q:2} 
	\end{align}
\end{subequations}
and correspondingly  its limiting system as  $\eps\rightarrow 0$ now is 
\beq
\label{eq:AP1:3.1.1}
\partial_t\rho=\langle v^2\rangle \partial_x u-\sigma_a \rho,\quad q=\partial_x \rho=\sigma_s u, \quad g=-vq/\sigma_s=-vu.
\eeq
The property \eqref{w:prop:0} has been  used. The introduction of $u$ is to deal with the spatially varying scattering coefficient $\sigma_s$. 
Note that the term $v\partial_x\rho$ in \eqref{eq:mmd:Q:2} can be replaced by $vq$.

\subsection{The IMEX1-LDG scheme}
\label{sec:num_sch}

To present the scheme,  we start with some notation.
For the computational domain $\Omega_x=[x_L, x_R]$ in space, a mesh,
 $x_L=x_{\frac{1}{2}}<x_{\frac{3}{2}}<\dots<x_{N+\frac{1}{2}}=x_R$, is introduced. Let
$I_i=[\xL, \xR]$ be an element, with $x_i$ as its center and $h_i$ as its length. Set $h=\max_i h_i$. ($\Dx$ in the introduction is just $h$ here.) For any nonnegative integer $k$, we define a finite dimensional discrete space
\begin{equation}
U_h^k=\left\{u\in L^2(\Omega_x): u|_{I_i}\in P^k(I_i), \forall i\right\},
\label{eq:DiscreteSpace:1mesh}
\end{equation}
where the local space $P^k(I)$ consists of polynomials of degree at most $k$ on $I$.  
We also introduce 
\begin{equation}
G_h^k=\left\{u(\cdot, v) \in U_h^k: \;\; \int_{\Omega_v}\int_{\Omega_x} |u(x,v)|^2dxdv<\infty\right\}.
\label{eq:DSpace:g}
\end{equation}
For a function $\phi\in U_h^k$, we write  $\phi(x^\pm)=\lim_{\Delta x\rightarrow 0^\pm} \phi(x+\Delta x)$,
and $\phi^\pm_\iR=\phi(x^\pm_\iR)$. The jump and average of $\phi$ at $x_\iR$ are defined as 
$[\phi]_\iR={\phi_\iR^+-\phi_\iR^-}$ and $\{\phi\}_\iR=\frac12(\phi_\iR^++\phi_\iR^-)$, respectively.

The IMEX1-LDG scheme in \cite{peng2018stability} involves a LDG  discretization in space and a first order globally stiffly accurate IMEX RK scheme in time. And an  IMEX strategy is adopted so that all the terms,  which are formally dominating in the regime $\eps\ll 1$, are treated implicitly. 
The IMEX1-LDG scheme for the model with a general $\sigma_s$ is based on the system \eqref{eq:mmd:Q}, and it is defined as below.
 Given  $\rho_h^{n}, \;q_h^{n},\;u_h^{n} \in U_h^k$, $g_h^{n}\in G_h^k$  
 that approximate the solution $\rho$, $q=\partial_x\rho$, $u$,
 and $g$ at $t^n$, we look for 
 $\rho_h^{n+1}, \;q_h^{n+1},\;u_h^{n+1} \in U_h^k$, $g_h^{n+1}\in G_h^k$
 at $t^{n+1}=t^n+\Dt$, such that $\forall\;   \testQ, \testU, \testR \in U_h^k$ and $ \testG\in G_h^k$,
\begin{subequations}
\label{eq:FDG:1T}
\begin{align}
\label{eq:FDG:1T:a}
&(q_h^{n+1},\testQ)+d_h(\rho_h^{n+1},\testQ)=0,\\
\label{eq:FDG:1T:c}
&(\sigma_s u_h^{n+1},\testU)=(q_h^{n+1},\testU),\\
\label{eq:FDG:1T:b}
&\big(\frac{\rho_h^{n+1}-\rho_h^n}{\Dt}, \testR\big)+l_h(\lgl v( g_h^n+\omega vu_h^n)\rgl,\testR)=\omega \langle v^2\rangle l_h(u_h^{n+1},\testR)
-\left(\sigma_a \rho_h^{n+1},\testR\right),\\
\label{eq:FDG:1T:d}
&\big(\frac{g_h^{n+1}-g_h^n}{\Dt},\testG\big)+\frac{1}{\eps} b_{h,v}(g_h^n,\testG)-\frac{v}{\eps^2} d_h(\rho_h^{n+1},\testG)
=-\frac{1}{\eps^2}(\sigma_s g_h^{n+1},\testG)
-\left(\sigma_a g_h^{n+1},\testG\right).
\end{align}
\end{subequations}
Here $(\cdot, \cdot)$ is the standard inner product for $L^2(\Omega_x)$.  
The bilinear forms $d_h, l_h$, and $b_{h, v}$ are all related to discrete spatial derivatives, and defined as
\begin{subequations}
\label{eq:bilinear:def}
\begin{align}
\label{eq:dh}
d_h(\rho_h, \testQ)&=\sum_i\int_{I_i} \rho_h \df_x\testQ dx + \sum_i \breve{\rho}_{h,\iL}[\testQ]_\iL , \\
\label{eq:lh}
l_h(u_h,\testR)&=-\sum_i \int_{I_i} u_h \partial_x\testR dx-\sum_i \hat u_{h,\iL} [\testR]_\iL,
\\
\label{eq:bh}
b_{h,v}(g_h,\testG)&=((\bI-\Pi)\mD_h(g_h; v), \testG) =(\mD_h(g_h; v) - \langle\mD_h(g_h; v)\rangle, \testG).
\end{align}
\end{subequations}
For a given $v\in \Omega_v$, the function $\mD_h(g_h; v)\in U_h^k$ in \eqref{eq:bh} is an upwind DG discretization of  the  transport term $v \df_x g$. It is determined by 
\begin{align}
(\mD_h(g_h; v), \testG)
=-\sum_i\left(\int_{I_i} vg_h \df_x\testG dx\right) - \sum_i \widetilde{(vg_h)}_\iL[\testG]_\iL, \quad \forall \psi\in U_h^k,
\label{eq:mD}
\end{align}
where $\widetilde{vg}$ is the upwind  flux,
\beq
\label{eq:vg:upwind:L-1}
\widetilde{vg}:=
\left\{
\begin{array}{ll}
v g^-,&\mbox{if}\; v>0\\
v g^+,&\mbox{if}\; v<0
\end{array}
\right.
=v\{g\}-\frac{|v|}{2}[g].
\eeq

The terms $\breve{\rho}$ and $\hat{u}$ in \eqref{eq:dh}-\eqref{eq:lh}  are one of the following alternating flux pair,
\beq 
\label{eq:flux}
\textrm{right-left: } \;\;\;\; \breve{\rho} = {\rho}^+,\;\;\hat{u} = u^-;\qquad
\textrm{left-right: } \;\;\;\;
\breve{\rho}={\rho}^-,\;\;\hat{u} = {u}^+. 
\eeq
The choice of the numerical fluxes $\breve{\rho}$ and $\hat{u}$ is important for the numerical solution to stay close to the local equilibrium when $\eps\ll 1$, and it contributes to the AP property of the scheme. Similar as in standard LDG methods, the auxiliary unknowns  $q_h$  and $u_h$  can be locally represented hence eliminated in terms of $\rho_h$. 

At $t=0$, the initialization is done via the $L^2$ projection $\pi_h$ onto $U_h^k$, namely, 
\beq
\label{eq:init0}
\rho_h^0(\cdot)=\pi_h \rho(\cdot, 0), \quad g_h^0(\cdot,v)=\pi_h g(\cdot,v,0),
\quad u_h^0(\cdot,v)=\pi_h (\sigma_s^{-1} \partial_x\rho).
\eeq

To complete the formulation of the scheme, one needs to specify the weight function $\omega$.  In our previous work \cite{peng2018stability}, Fourier-type stability analysis suggests 
that  $\omega$ should be chosen in the form of $\omega=\omega(\frac{\eps}{h},\frac{\eps^2}{\Dt})$, to preserve the intrinsic scale of the underlying model. 
In this paper, we only consider  $\omega=\omega(\eps/(\sigma_m h))$,  which is independent of $\eps^2/\Dt$. Some specific examples include $\omega=\exp\big({-\eps}/({\sigma_m h}) \big)$ and $\omega\equiv 1$. One can also use a piecewise constant choice $\omega={\bf 1}_{\{ \eps/(\sigma_m h) \leq \alpha\}}$, with some fixed positive constant  $\alpha$, see Remark \ref{rem:stab:w} for a specific choice of  $\alpha$ recommended by our stability analysis. (Here  ${\bf 1}_D$ is an indicator function with respect to a set $D$.) Note that all these choices are non-negative and independent of $x$, satisfying \eqref{w:prop:0}.  

The next lemma states the relation of bilinear forms $d_h$ and $l_h$, and this can be verified directly.
\begin{lem}
\label{lem:1}
With either alternating flux pair in \eqref{eq:flux},  the bilinear forms $b_h$ and $l_h$ are related, 
\begin{align}
l_h(\testQ,\testR)= d_h(\testR,\testQ),  \qquad \forall \varphi, \;  \phi\in U_h^k.
\label{eq:lem1:1}
\end{align}
\end{lem}

The unique solvability of the solution to the IMEX1-LDG method is given in next  proposition, together with some properties in \eqref{eq:prop:1} that can be easily  verified.
 The key to prove the first part of the proposition is the unique solvability of the problem examined in Lemma \ref{lem:sol}. 
\begin{prop}\label{prop:sol}
The IMEX1-LDG method is uniquely solvable for any $\eps\geq 0$. In addition, the solution satisfies
\begin{equation}
\langle g_h^{n} \rangle=0, \; \forall n\geq 0,\qquad 
(\sigma_s u_h^m, \testU)=-l_h(\testU, \rho_h^m), \; \forall \testU\in U_h^k, \;\; \forall m\geq 1.
\label{eq:prop:1}
\end{equation}
\end{prop}

\begin{lem}\label{lem:sol}
Given $S\in L^2(\Omega_x)$ and $\gamma_j\geq 0, j=1, 2$. Consider the following problem: look for $\rho_h, q_h, u_h\in U_h^k$, such that $ \forall \testQ, \testU, \testR\in U_h^k$,
\begin{equation}
(q_h, \testQ)+d_h(\rho_h, \testQ)=0,  \quad (\sigma_s u_h, \testU)=(q_h, \testU), \quad (\rho_h, \testR)-\gamma_1 l_h(u_h, \testR)=-\gamma_2(\sigma_a \rho_h, \testR)+(S, \testR).
\label{eq:lem10:1}
\end{equation}
Then $\rho_h, q_h, u_h$ are uniquely solvable. 
\end{lem}
\begin{proof}
We first consider the homogeneous case with $S=0$. Take  $\testQ=\testU=u_h, \testR=\rho_h$, use the relation of $d_h$ and $l_h$, we get
$$(\rho_h, \rho_h)+\gamma_1(\sigma_s u_h, u_h)+\gamma_2(\sigma_a \rho_h, \rho_h)=0. $$
With $\gamma_1, \gamma_2, \sigma_s, \sigma_a$ being non-negative, one has $\rho_h=0$, and the equations in \eqref{eq:lem10:1} further ensure $q_h=u_h=0$. 
This, in combination with the linearity of the problem  as well as that both the solution and the test function are from the same finite dimensional space $U_h^k$, implies the unique solvability of the problem with the general  source term $S$.
\end{proof}

Following the formal asymptotic analysis as in \cite{peng2018stability}, we can show the IMEX1-LDG method is AP, namely as $\eps\rightarrow 0$, its limiting scheme is a consistent and stable discretization of the limiting system \eqref{eq:AP1:3.1.1}, when the initial data is well-prepared. This will be stated in Section \ref{sec:ap} and proved in Section \ref{sec:ap:proofs} once the uniform stability is available. When the initial data is not well-prepared, our scheme can adopt a similar initial fix  \cite{peng2018stability} when $n=0$ to stay AP. There is no change to numerical stability, while the AP property can be established rigorously and the details are not presented in this paper.

%################################
\subsection{Norms, inverse inequalities, and more notation}
\label{sec:prelim}

We introduce some standard norms $||\phi||=||\phi||_{L^2(\Ox)}$, $|||\phi|||=(\langle ||\phi||^2\rangle)^{1/2}$,  and weighted norms
$||\phi||_s= ||\sqrt{\sigma_s}\phi||$, $|||\phi|||_s =|||\sqrt{\sigma_s}\phi|||.$
For a bounded function $\psi(v)$ of $v$,  without confusion we will write $||\psi||_\infty=||\psi||_{L^\infty(\Ov)}$.  Even though for our specific examples with $\Ov=[-1, 1]$ or $\{-1, 1\}$, we have $||v||_\infty=||v^2||_\infty=1$, we still keep $||v||_\infty$ and $||v^2||_\infty$ in most results, to possibly inform about the case with a more general bounded velocity space $\Ov$.

In our analysis, the following inverse inequalities will be frequently used, and they are fairly standard in finite element analysis: there exist constants $C_{inv}=C_{inv}(k)$ and $\widehat{C}_{inv}=\widehat{C}_{inv}(k)$, such that for any $\phi \in P^k([a,b])$,
\begin{subequations}
\label{eq:inv}
\begin{align}
&|\phi(y)|^2 (b-a) \leq C_{inv}\int_a^b |\phi(x)|^2 dx, \qquad\mbox{with}\; y=a \;\mbox{or}\; b \label{eq:inv:1},\\
&(b-a)^2 \int_a^b |\phi'(x)|^2 dx \leq \widehat{C}_{inv}\int_a^b |\phi (x)|^2 dx.
\label{eq:inv:2}
\end{align}
\end{subequations}
Particularly, $C_{inv}(k)|_{k=0}=1$. Next lemma states a property of the inverse constants $\widehat{C}_{inv},  C_{inv}$.

\begin{lem} With $\Omega_v=[-1,1]$ or $\Omega_v=\{-1,1\}$, and with  $\widehat{C}_{inv}$, $C_{inv}$ from \eqref{eq:inv},  we define 
\beq
\label{def:K}
\mK=\mK(k)=\frac{8(\bbb)^2}{\aaa}=\frac{8(C_{inv})^2}{\widehat{C}_{inv}}.
\eeq 
 Then at least for $k=1, 2, \cdots, 9$, we have $\mK>1.$
 \label{lem:mK}
\end{lem}
\begin{proof}
Based on Lemmas 1-2 in \cite{Matt_Li} and a linear scaling, one can take $C_{inv}=(k+1)^2$ and $ \widehat{C}_{inv}=12k^4$, which can be used to verify $\mK>1$ directly for $k=1, 2, \cdots 9$.
\end{proof}
 Sharper values of $C_{inv}(k)$ and $\widehat{C}_{inv}(k)$ can be numerically obtained for {\em each} $k$ by solving an eigenvalue problem (see  Section 4.1 in \cite{Matt_Li}), hence  one can check numerically whether $\mK>1$ holds or not for larger $k$. Given the temporal accuracy of the IMEX1-LDG method is first order, it is more than enough for us to consider $k\leq 9$ in our analysis.

  For convenient  reference, we summarize  in Table \ref{tab:notation} the definitions of some notation arising from analysis, including $\lambda_\star, \widehat{\lambda}_\star$ and $\mu_\star$, which all depend on inverse constants 
 hence on $k$.  They also depend on the  weight function $\omega$ and the velocity space $\Omega_v$. 
 The same table also includes the definitions of 
 $\mK$ in \eqref{def:K},  a function $\mu_S(\lambda)$ and its inverse $\lambda_S(\mu)$, as well as two more functions $\lambda_j(\mu), j=1, 2$. The place where each notation appears for the first time is also included.

\begin{table}[htb]
	\caption{Some notation  (with the possible $\omega$-dependence suppressed) and the place of the first appearance}
		\begin{center}
		\begin{tabular}{lc}
					\hline	
		notation & the first appearance \\
			\hline		
			$\mK=\frac{8(\bbb)^2}{\aaa} $  & \eqref{def:K} \\\hline							
$\lambda_\star=\frac{2(1-1/(2\omega) )C_{inv}||v||_\infty}{\widehat{C}_{inv}||v^2||_\infty+8(C_{inv}||v||_\infty)^2}$ & \eqref{eq:thm:stab3:o}\\\hline
			$\mu_\star=\frac{1+\frac{1}{2\omega}\mK }{1+\mK}=\frac{\widehat{C}_{inv}||v^2||_\infty+4(C_{inv}||v||_\infty)^2/\omega}{\widehat{C}_{inv}||v^2||_\infty+8(C_{inv}||v||_\infty)^2}$ & \eqref{eq:thm:stab5:o:1}\\\hline	
			$\mu_S(\lambda)=\frac{1}{2\omega}+\frac{1}{2}\lambda\frac{\aaa}{\bbb}$
			&\eqref{eq:thm:stab5:o:3} \\ \hline
$\lambda_S(\mu)=\mu_S^{-1}(\mu)=2(\mu-\frac{1}{2\omega})\frac{\bbb}{\aaa}$ & Lemma \ref{lem:muS:mono}\\\hline	
$\widehat{\lambda}_\star=\lambda_S(1)=2(1-\frac{1}{2\omega})\frac{\bbb}{\aaa}$ & \eqref{eq:thm:stab5:o:3}\\\hline	
$\lamone(\mu)=\sqrt{\frac{(1-\mu)(\mu-\frac{1}{2\omega}) }{2\widehat{C}_{inv} ||v^2||_\infty} }, \quad \lamtwo(\mu)=\frac{1-\mu}{4C_{inv}||v||_\infty} $ & \eqref{eqn:lam12}\\\hline
		\end{tabular}
		\end{center}
		\label{tab:notation}
\end{table}

%%%%%%%%%%%%%%%%%%%
% Stability Analysis
%%%%%%%%%%%%%%%%%%%
\section{Numerical stability}
\label{sec3:stab}

In this section, we will establish numerical stability for the IMEX1-LDG method following {\em an energy approach}. At the continuous level, one can derive an energy relation
\beq
\label{eq:stab:c}
\frac{1}{2} \frac{d}{dt}\Big(||\rho||^2+\eps^2|||g|||^2\Big)=-\int_\Ov\int_\Ox \Big(\sigma_s g^2+\sigma_a (\rho+\eps g)^2\Big) dxdv
\eeq
for the model \eqref{eq:mmd:LJ0},
implying the energy $|||f|||^2(t)=||\rho||^2(t)+\eps^2|||g|||^2(t)$ does not grow in time.  Our numerical stability is a discrete analogue. Particularly,  we want to confirm that the method is unconditionally stable in the diffusive regime when $\eps\ll 1$ and it is uniformly stable in $\eps$, 
with a general form of the weight function $\omega=\omega(\eps/(\sigma_m h))$  taken into account. 
Without loss of generality, we assume the mesh is uniform with $h=h_i, \forall i$. Our results can be extended to general meshes when $\frac{\max_i h_i}{\min_i h_i}$ is bounded uniformly during mesh refinement.  For easy  readability, we will present and discuss the main results in this section, and defer the proofs   to Sections \ref{sec:stab:proofs}-\ref{sec:stab:proofs:opt}.

The natural first attempt is to follow a similar analysis as in \cite{jang2014analysis}, and this will lead to the stability result in next theorem.

%%%%%%%%%%%%
% stability without mu
%%%%%%%%%%%%
\begin{thm}\label{thm:stab_nmu} 
The following stability result holds for the IMEX1-LDG method, defined  as \eqref{eq:FDG:1T} with \eqref{eq:bilinear:def}-\eqref{eq:flux}, 
\begin{align}
E_h^{n+1}\leq E_h^{n}, \;\;\forall n\geq 1, \;\; \textrm{with} \; \; E_h^n:=||\rho_h^n||^2+\eps^2|||g_h^{n-1}|||^2+\Dt\omega\lgl v^2\rgl ||u_h^n||_s^2,
\label{eq:energy_nmu}
\end{align}
under the time step condition,
 \beq
 \label{eq:cfl}
 \Dt\leq \Dt_{stab}=\left\{
  \begin{array}{ll}
     \frac{2h}{\alpha_2 \alpha_3}(\sigma_m h+\alpha_3\eps), &\mbox{for}\; k=0,\\
  \frac{h}{\alpha_1+\alpha_2\alpha_3} (\sigma_mh+\min(\eps, \frac{\alpha_2h}{\alpha_1})\alpha_3), &\mbox{for}\; k\geq 1.
  \end{array}
 \right.
 \eeq
Here $\alpha_i, i=1, 2, 3$ are defined in terms of the inverse constants and the velocity space, namely,
\beq
\label{eq:alpha}
\alpha_1=(||v||_\infty^2+\langle v^2\rangle) \widehat{C}_{\textrm{inv}}\;,\quad\alpha_2= 2(||v||_\infty+\langle|v|\rangle)C_{\textrm{inv}}\;,\quad
\alpha_3=2 ||v||_\infty C_{\textrm{inv}}.
\eeq
\end{thm}
Note that the time step condition in \eqref{eq:cfl} is essentially the same as the one for the DG-IMEX1 method defined in \cite{jang2014analysis}.  This theorem, on one hand, gives uniform stability  with respect to $\eps$,  which is important for the AP property of the method, see Section \ref{sec:ap} and Section \ref{sec:ap:proofs}, also \cite{jang2014analysis}. On the other hand, the theorem {\em  fails} to capture the unconditional stability property of the method in the diffusive regime when  $\eps\ll 1$.  

The main reason that Theorem \ref{thm:stab_nmu} missed the unconditional stability we observed numerically and predicted by Fourier analysis in \cite{peng2018stability} is that the damping mechanism associated with the scattering operator (see the right hand side term in \eqref{eq:stab:c})  has not been fully utilized in the analysis. By better exploring the contribution of the scattering operator, new   stability results can be established and they will  capture the unconditional stability  property of the method.  This indeed is one main contribution of this work.  The new stability analysis will be based on a new discrete energy $E_{h, \mu}^n$.

\begin{defn} %({\bf $\mu$-stable})
\label{def:mu-stable}
For any given  constant $\mu\in [0, 1]$, we define a discrete energy
\begin{align}
E_{h,\mu}^n= ||\rho_h^n||^2+\eps^2|||g_h^{n-1}|||^2+\omega\Dt \lgl v^2\rgl ||u_h^n||_s^2+\Dt(1-\mu)||| g_h^{n-1}|||_s^2.
\label{eqn:energy-mu}
\end{align} 
The IMEX1-LDG method is said to be {\em $\mu$-stable} if it satisfies
 \begin{align}
 E_{h,\mu}^{n+1}\leq E_{h,\mu}^n,\qquad \forall n\geq 1.
 \label{eq:mu:stab}
  \end{align}
If the method is $\mu$-stable for some $\mu\in[0,1]$, then it is said to be {\em stable}. If the scheme being $\mu$-stable (resp. stable) is independent of the time step size $\Dt$, the method is further said to be {\em unconditionally $\mu$-stable} (resp. {\em unconditionally stable}). Note that $E_{h, 1}^n=E_n^n$.
\end{defn}

With respect to the $\mu$-stability above, a new stability result will be stated in next theorem under the assumption  $\omega>1/2$. When the weight function is $\omega\equiv 1$, this assumption always holds. In general, with the property $\omega\rightarrow 1$ as $\eps\rightarrow 0$ in \eqref{w:prop:0}, the stability result can at least capture the property of the method in the diffusive regime.  
%
%
%%%%%%%%%%%%
% mu-stability
%%%%%%%%%%%%
\begin{thm}{\bf ($\mu$-stability: $\omega>\frac{1}{2}$)}
\label{thm:mu} When 
$\omega>\frac{1}{2}$,
 the following $\mu$-stability results hold for the IMEX1-LDG method, defined as  \eqref{eq:FDG:1T} with \eqref{eq:bilinear:def}-\eqref{eq:flux}. 
 \begin{itemize}
 \item [(i)]  When $k=0$ and with any fixed  $\mu\in[\frac{1}{2\omega},1]$,  if
\beq
\label{eq:thm:stab1}
\frac{\eps}{\sigma_m h}\leq  \lamzero(\mu):=\frac{1-\mu}{2C_{inv}||v||_\infty }=\frac{1-\mu}{2||v||_\infty },
\eeq
the IMEX1-LDG method is unconditionally $\mu$-stable.
  Otherwise,
 the method is conditionally $\mu$-stable when the time step satisfies
 \beq
\label{eq:thm:stab2}
\Dt\leq\tau_{\eps, h, 0}(\mu):=\frac{2\eps^2 h}{2 C_{inv}||v||_\infty \eps-(1-\mu)\sigma_m h}=\frac{2\eps^2 h}{2 ||v||_\infty \eps-(1-\mu)\sigma_m h}.
\eeq
Here we have used $C_{inv}(k)|_{k=0}=1$. The result can be expressed  more compactly as $\Dt\leq \tauzero(\mu)$, by introducing an extended real-valued function
\begin{equation}
\tauzero(\mu)=\begin{cases}
                         \infty,    \quad  &\text{if}\;\frac{\eps}{\sigma_m h} \leq \lambda_0(\mu),\\
		        \tau_{\eps,h,0}(\mu)   \quad &\text{otherwise}.
		        \end{cases}
\label{eqn:Gtau0}
\end{equation}
And the scheme is unconditionally $\mu$-stable if and only if $\tauzero(\mu)=\infty$.

\item [(ii)]  When $k\ge 1$ and with any fixed  $\mu\in(\frac{1}{2\omega},1]$, if
 \beq
\label{eq:thm:stab3}
\frac{\eps}{\sigma_m h}\leq \min\left(\lamone(\mu), \lamtwo(\mu)\right),
\eeq
the IMEX1-LDG method is unconditionally $\mu$-stable. 
Otherwise, the method is conditionally $\mu$-stable when the time step satisfies
\beq
\label{eq:thm:stab4}
\Dt\leq\left\{
\begin{array}{ll}
\tau_{\eps,h,1}(\mu),&\textrm{if}\;
\lamone(\mu)<\frac{\eps}{\sigma_m h}\leq\lamtwo(\mu),\\
\tau_{\eps,h,2}(\mu),&\textrm{if}\;
\lamtwo(\mu)<\frac{\eps}{\sigma_m h}\leq\lamone(\mu),\\
\min(\tau_{\eps,h,1}(\mu), \tau_{\eps,h,2}(\mu)),&\textrm{if}\;\frac{\eps}{\sigma_m h}\geq \max\left(\lamone(\mu), \lamtwo(\mu)\right).
\end{array}
\right.
\eeq
Here
\begin{subequations}
\label{eq:tau12}
\begin{align}
\lamone(\mu)&:=\sqrt{\frac{(1-\mu)(\mu-\frac{1}{2\omega} ) } {2\widehat{C}_{inv}||v^2||_\infty} }, \qquad \lamtwo(\mu):=\frac{1-\mu}{4C_{inv}||v||_\infty }, \label{eqn:lam12}\\
\tau_{\eps,h,1}(\mu)&:=\frac{ 2\eps^2(\mu-\frac{1}{2\omega} )h^2 \sigma_m} {2\eps^2 \widehat{C}_{inv}||v^2||_\infty-(1-\mu)(\mu-\frac{1}{2\omega} )\sigma_m^2h^2 },
\label{eqn:tau1}\\
\tau_{\eps,h,2}(\mu)&: =\frac{2\eps^2 h}{4C_{inv}||v||_\infty\eps-(1-\mu)\sigma_mh} .\label{eqn:tau2}
\end{align}
\end{subequations}
Again the results can be expressed more compactly as $\Dt\leq \min\left(\tauone(\mu), \tautwo(\mu)\right)$, by introducing two  extended real-valued functions
\begin{equation}
\label{eq:Gtau12}
\widehat{\tau}_{\eps, \mu, i}(\mu)=\begin{cases}
                         \infty,    \quad  &\text{if}\;\frac{\eps}{\sigma_m h} \leq \lambda_i(\mu)\\
		        \tau_{\eps,h,i}(\mu),
		         \quad &\text{otherwise}
		        \end{cases},
		        \qquad i=1, 2. 
\end{equation}
And the scheme is unconditionally $\mu$-stable if and only if $\min\left(\tauone(\mu), \tautwo(\mu)\right)=\infty$. 
 \end{itemize}
 \end{thm}
 
We can see now that with some choice of  $\mu$, this new stability result in Theorem \ref{thm:mu} captures the unconditional stability in the diffusive regime. This regime at the discrete level is characterized by \eqref{eq:thm:stab1} and \eqref{eq:thm:stab3} when $\eps/(\sigma_m h)$ is relatively small.  It is also clear that the choice of $\mu$ matters when one  interprets  the results. For instance when $k=0$, with $\mu=1/(2\omega)$, the IMEX1-LDG method is unconditionally stable in the diffusive regime, yet with $\mu=1$, we  no longer see this property according to Theorem  \ref{thm:mu}.  This motivates us to further refine the results.   Based on the definition of the (unconditional) stability in Definition \ref{def:mu-stable},  we consider an optimization problem for any given $\eps, h$, and look for  the ``best'' possible choice of $\mu$, that maximizes   the unconditionally stable region (that is, to maximize the allowable range of  $\eps/(\sigma_m h)$ in  \eqref{eq:thm:stab1} and \eqref{eq:thm:stab3}), and possibly  also maximizes the allowable time step condition in  \eqref{eq:thm:stab2} and \eqref{eq:thm:stab4}  when the method is conditionally stable. 
The optimization process leads to Theoreom \ref{thm:stab} that comes next, with the underlying logic as 
%and the logical foundation is 
$$\max\{\lambda: \lambda\leq \Theta(\mu, \lambda), \forall \mu\in [\mathcal{H}(\lambda), 1] \} =\max\{\lambda:  \lambda\leq \max_{\mu\in [\mathcal{H}(\lambda), 1] }\Theta(\mu, \lambda)\},
$$
if all maximums are assumed to exist, and $\Theta, \mathcal{H}$ are some continuous functions. The relation holds if $[\mathcal{H}(\lambda), 1]$ is replaced by $(\mathcal{H}(\lambda), 1]$. 
Note that the weight function in the stability results is in the form $\omega=\omega(\eps/(\sigma_m h))$.

\begin{thm}{\bf (Stability: $\omega>\frac{1}{2}$)}
\label{thm:stab}  When 
$\omega>\frac{1}{2}$,
 the following stability results hold for the IMEX1-LDG method, defined as  \eqref{eq:FDG:1T} with \eqref{eq:bilinear:def}-\eqref{eq:flux}. 

 \begin{itemize}
 \item [(i)]  When $k=0$,  the IMEX1-LDG method is stable when 
 \beq
 \Dt\leq \dtstabzero(\eps, h):= \max_{\mu\in[\frac{1}{2\omega}, 1]} \tauzero(\mu)
 =\tauzero\left(\frac{1}{2\omega}\right).
 \eeq
 In particular, the method is unconditionally stable if $\dtstabzero(\eps, h)=\infty$, that is, when 
\beq
\label{eq:thm:stab1:o}
\frac{\eps}{\sigma_m h}\leq  \max_{\mu\in[\frac{1}{2\omega}, 1]} \lamzero(\mu)=\lamzero\left(\frac{1}{2\omega}\right)
= \frac{1-\frac{1}{2\omega}}{2||v||_\infty }.
\eeq
  Otherwise,
 the method is conditionally stable under the time step condition 
 \begin{align}
\label{eq:thm:stab2:o}
\Dt\leq  \max_{\mu\in[\frac{1}{2\omega}, 1]} \tau_{\eps, h, 0}(\mu)= \tau_{\eps, h, 0}\left(\frac{1}{2\omega}\right)
=\frac{2\eps^2 h}{2 ||v||_\infty \eps-(1-\frac{1}{2\omega})\sigma_m h}.
\end{align}

\item [(ii)]  
When $1\leq k \leq 9$,  the IMEX1-LDG method is stable when 
 \beq
 \Dt\leq \dtstabk(\eps, h):= \max_{\mu\in(\frac{1}{2\omega}, 1]} \min\left( \tauone(\mu), \tautwo(\mu) \right).
 \eeq
 In particular, the method is unconditionally stable if $\dtstabk(\eps, h)=\infty$, that is
when 
\begin{align}
\frac{\eps}{\sigma_mh}&\leq
\max_{\mu\in (\frac{1}{2\omega}, 1]} \min\left(\lamone(\mu), \lamtwo(\mu)\right)=
\min\left(\lamone(\mu), \lamtwo(\mu)\right)|_{\mu=\mu_\star}\notag\\
&=
  \lambda_\star:=\frac{2(1-\frac{1}{2\omega} )C_{inv}||v||_\infty}{\widehat{C}_{inv}||v^2||_\infty+8(C_{inv}||v||_\infty)^2}.
\label{eq:thm:stab3:o}
\end{align}
Otherwise the method is conditionally stable under the time step condition
\begin{align}
\Dt&\leq
\max_{\mu\in (\frac{1}{2\omega}, 1]}\min\left(\tauone(\mu), \tautwo(\mu)\right)\notag\\
&= \tau_{\eps, h, 1}( \min(\mu_S(\frac{\eps}{\sigma_m h}), 1)\notag\\
&=\left\{
\begin{array}{ll}
\tau_{\eps, h, 1}(\mu_S(\frac{\eps}{\sigma_m h}))=
\frac{4\bbb\eps^2 h}{(8(\bbb)^2+\aaa)\eps-2\bbb(1-\frac{1}{2\omega})\sigma_m h}, & \text{for}\;\lambda_\star<\frac{\eps}{\sigma_m h}\leq \widehat{\lambda}_\star,\\
\tau_{\eps, h, 1}(1)=\frac{(1-\frac{1}{2\omega})\sigma_m h^2}{\aaa},
			&\text{for}\; \frac{\eps}{\sigma_m h} > \widehat{\lambda}_\star.			
\end{array}
\right.
\label{eq:thm:stab4:o}
\end{align}
Here 
\begin{subequations}
\begin{align}
\mu_\star&=\frac{1+\frac{1}{2\omega}\mK }{1+\mK}=\frac{\widehat{C}_{inv}||v^2||_\infty+4(C_{inv}||v||_\infty)^2/\omega}{\widehat{C}_{inv}||v^2||_\infty+8(C_{inv}||v||_\infty)^2},\label{eq:thm:stab5:o:1}\\
\mu_S(\lambda)&=\frac{1}{2\omega}+\frac{1}{2}\lambda\frac{\aaa}{\bbb}, \quad
%\label{eq:thm:stab5:o:2}\\
\widehat{\lambda}_\star=\mu_S^{-1}(1)=2(1-\frac{1}{2\omega})\frac{\bbb}{\aaa}.\label{eq:thm:stab5:o:3}
\end{align}
\end{subequations}
 \end{itemize}
  \end{thm}

\begin{rem} The results in Theorem \ref{thm:stab} also implies an alternative route to obtain this theorem. In fact,  one can establish  Theorem \ref{thm:stab} by following the proof of Theorem \ref{thm:mu} and taking $\mu=\frac{1}{2\omega}$ when $k=0$, and taking 
\begin{align}
&\mu=\mu(\eps, h; k):
=\left\{
\begin{array}{ll}
\mu_\star, & \text{for}\;\frac{\eps}{\sigma_m h}\leq \lambda_\star\\
\min\left(\mu_S(\frac{\eps}{\sigma_m h}),1\right), & \text{for}\;\frac{\eps}{\sigma_m h}>\lambda_\star,\\
\end{array}
\right.
\end{align}
in  defining the discrete energy $E_{h,\mu}^n$ in \eqref{eqn:energy-mu}, tailored for each given $\eps$, $h$ (implicitly also for a given  weight function $\omega(\eps/(\sigma_m h))$.
Note that $\mu$ is chosen according to $\eps/(\sigma_m h)$ that describes the regime the model is in with respect to the 
discretization parameter $h$. The assumption $1\leq k\leq 9$ in this theorem is to ensure $\mK>1$, see Lemma \ref{lem:mK}.
\end{rem}

Following the notion of the stability in Definition \ref{def:mu-stable} and with  $E_{h, 1}^n=E_h^n$, we can combine the results in Theorem \ref{thm:stab_nmu} and Theorem \ref{thm:stab}, and obtain our final results on numerical stability for a general weight function $\omega=\omega(\eps/(\sigma_m h))$ that satisfies the property \eqref{w:prop:0}.

\begin{thm}\label{thm:stab_final}
The  following stability results hold for the IMEX1-LDG method, defined as  \eqref{eq:FDG:1T} with \eqref{eq:bilinear:def}-\eqref{eq:flux}. 

 \begin{itemize}
 \item[(i)] 
When $k=0$, the method is unconditionally stable, if 
\begin{equation}
\omega>\frac{1}{2}\quad \text{and}\quad \frac{\eps}{\sigma_m h}\leq \frac{1-\frac{1}{2\omega}}{2||v||_\infty}.
\label{eq:thm:stab1:f}
\end{equation}
Otherwise, the method is conditionally stable under the time step condition
\begin{equation}
\Dt\leq \max\left(
\frac{2||v||_\infty \eps h+\sigma_m h^2}{ 2||v||_\infty (||v||_\infty+\lgl |v|\rgl)},\;
 \frac{ 2\eps^2 h \cdot{\bf 1}_{\{\omega>\frac{1}{2}\}}}{2 ||v||_\infty \eps-(1-\frac{1}{2\omega})\sigma_m h}
\right).
\label{eq:thm:stab2:f}
\end{equation}

\item[(ii)] When $1\leq k \leq 9$,  the method is unconditionally stable, if 
\begin{equation}
\omega>\frac{1}{2}\quad \text{and}\quad
\frac{\eps}{\sigma_m h}\leq  \lambda_\star.
\label{eq:thm:stab3:f}
\end{equation}
Otherwise, the method is conditionally stable under the time step condition
\begin{equation}
\Dt\leq \max\left(  
\frac{h}{\alpha_1+\alpha_2\alpha_3} (\sigma_m h+\min(\eps, \frac{\alpha_2h}{\alpha_1})\alpha_3),\;
 {\bf{1} }_{\{\omega>\frac{1}{2}\} }\cdot \tau_{\eps,h,1}\big(\min(\mu_S(\frac{\eps}{\sigma_m h}),1) \big)
 \right),
 \label{eq:thm:stab4:f}
\end{equation}
where $\alpha_i, i=1, 2, 3$ are given in \eqref{eq:alpha}.
\end{itemize}
\end{thm}
%%%%%%%%%%%
% Discussion on the stability results for k=0
%%%%%%%%%%%%%

\begin{rem} 
\label{rem:stab:w}  When $k=0$, the IMEX1-LDG method, denoted as IMEX1-LDG1 method,  will be of first order in both space and time. We here will examine more  explicitly the stability results for this first order method   when the model is the telegraph equation (referred to as T model) and the one-group transport equation in slab geometry (referred to as OG model). Note that $\lgl |v|\rgl =1$ for the former, and $\lgl |v|\rgl =\frac{1}{2}$ for the latter.  Particularly, we want to give the results for  three  weight functions, including $\omega\equiv 1$ and $\omega=\exp(-\frac{\eps}{\sigma_m h})$ (used in \cite{peng2018stability}), and a piecewise-defined $\omega$ takings value $1$ for ``relatively small'' $\eps$ and $0$ for large $\eps$ (used in \cite{boscarino2014high}).  Our analysis will provide some guidance on how to define such  piecewise constant $\omega$. All three examples of $\omega$ are monotonically non-increasing in $\eps/(\sigma_m h)$.
First of all, for the IMEX1-LDG1 method, the result \eqref{eq:thm:stab2:f} is indeed   
\begin{equation}
\Dt\leq \max\left(
\frac{2\eps h+\sigma_m h^2}{\beta},\;
 \frac{ 2\eps^2 h \cdot{\bf 1}_{\{\omega>\frac{1}{2}\}}}{2 \eps-(1-\frac{1}{2\omega})\sigma_m h}
\right),
\quad
\beta=\left\{
\begin{array}{ll}
4 & \textrm{(T model)}\\
3 & \textrm{(OG model)}
\end{array}
\right..
\label{eq:thm:stab5:f}
\end{equation}

\begin{itemize}
\item[i.)]  We first consider $\omega\equiv 1$. It is easy to verify that $ \frac{ 2\eps^2 h}{2 \eps-(1-\frac{1}{2\omega})\sigma_m h}\Big|_{\omega=1}\geq  \frac{2\eps h+\sigma_m h^2}{\beta}$ always holds.  Then the stability results for the IMEX1-LDG1 method in  \eqref{eq:thm:stab1:f}-\eqref{eq:thm:stab2:f} become: the method is unconditionally stable when $\eps/(\sigma_m h)\leq 1/4$, otherwise it is conditionally stable under the time step condition 
$\Dt\leq \frac{4\eps^2 h}{4\eps-\sigma_m h}.$
Note that this stability condition 
is the same for both T and OG models, and is used in  \cite{peng2018stability} for numerical experiments.

\item[ii.)] We next consider a piecewise constant $\omega$, taking value either $1$ or $0$. To have the largest possible unconditional  stability region, our analysis suggests 
$
\omega={\bf 1}_{\{\eps/(\sigma_m h)\leq 1/4\}},
$
and the respective stability results for the IMEX1-LDG1 method become: the method is unconditionally stable when $\eps/(\sigma_m h)\leq 1/4$, and it is conditionally stable when
\begin{equation}
\Dt\leq \frac{2\eps h+\sigma_m h^2}{\beta}.
\label{eq:thm:stab7:f}
\end{equation}
Note when $\omega=0$, our IMEX1-LDG1 method is just the DG1-IMEX1 method in \cite{AP1, jang2014analysis}, with \eqref{eq:thm:stab7:f} as the respective time step condition for stability. The results imply that, if we apply the IMEX1-LDG1 method with $\omega=1$ in the relatively diffusive regime, namely $\eps/(\sigma_m h)\leq 1/4$, and apply the DG1-IMEX1 method elsewhere, the stability condition  will be inherited from the method used in each regime. %This  intuitive yet not so obvious fact is confirmed  by our analysis. 
\item[iii.)]  The final case is for $\omega=\exp(-\eps/(\sigma_m h))$. Note that  $\omega>1/2$ is equivalent to $\eps/(\sigma_m h)<r_*$ with $r_*=\ln(2)\approx 0.69314718$, and  the second inequality in \eqref{eq:thm:stab1:f} is equivalent to $\eps/(\sigma_m h)\leq r_\dag$, where $r_{\dag}\approx  0.19589899$ is the root of $x=(2-e^x)/4$.  While the stability results in \eqref{eq:thm:stab1:f}-\eqref{eq:thm:stab2:f} are straightforward when $\eps/(\sigma_m h)\leq r_\dag$ and when $\eps/(\sigma_m h)\geq r_*$, the results when $\eps/(\sigma_m h)\in (r_\dag, r_*)$ would depend on the model.  
With some calculation, one can obtain the stability results for the 
IMEX1-LDG1 method with this weight function,
\begin{equation}
\textrm{T model}: \quad \Dt\leq \left\{
\begin{array}{ll}
\infty & \text{when} \; \eps/(\sigma_m h)\leq r_\dag\\
 \frac{ 2\eps^2 h }{2 \eps-\Big(1-\exp(\eps/(\sigma_m h))/2\Big)\sigma_m h} &  \text{when} \; \eps/(\sigma_m h)\in (r_\dag, r_*)\\
(2\eps h+\sigma_m h^2)/4 & \text{when} \; \eps/(\sigma_m h)\geq r_*
\end{array}
\right., 
\end{equation}
\begin{equation}
\textrm{OG model}: \quad \Dt\leq \left\{
\begin{array}{ll}
\infty & \text{when} \; \eps/(\sigma_m h)\leq r_\dag\\
 \frac{ 2\eps^2 h }{2 \eps-\Big(1-\exp(\eps/(\sigma_m h))/2\Big)\sigma_m h} &  \text{when} \; \eps/(\sigma_m h)\in (r_\dag, r_\circ)\\
(2\eps h+\sigma_m h^2)/3 & \text{when} \; \eps/(\sigma_m h)\geq r_\circ
\end{array}
\right. . 
\end{equation}
Here $r_\circ\approx 0.38161849$ is the root of $(2x+1)/3=2x^2/\Big(2x-1+\exp(x)/2\Big)$. 
\end{itemize}
\end{rem}

%%%%%%%%%%%%%%%%%%%%%%%%%%%%%%%%%%%%%%%%%%%%%%%%%%%%%%%%%%%%%%%%%%
%%
%% Absorption terms
%%
%%%%%%%%%%%%%%%%%%%%%%%%%%%%%%%%%%%%%%%%%%%%%%%%%%%%%%%%%%%%%%%%%%

\section{Asymptotic preserving (AP) property}
\label{sec:ap}

In this section, we will state the main theorem on the AP property of the IMEX1-LDG method when the initial data is well-prepared, namely, $g+v\df_x\rho/\sigma_s=O(\varepsilon)$ at $t=0$. The proof will be established in Section \ref{sec:ap:proofs} based on uniform stability property of the method. 
With $W=\rho, q, g, u$,  we write $W_\varepsilon|_{t=0}=W_\varepsilon^0$, $W|_{t=0}=W_0$, and denote the numerical solution at time $t^n$ as $W_\AP^n$ to emphasize the dependence on $h$, $\Dt$, $\varepsilon$.
Here
$q^0_\varepsilon=\df_x\rho_\eps^0$ and $q_0=\df_x\rho^0$
are   weak derivatives of $\rho^0_\varepsilon$ and $\rho_0$, respectively.
The following assumptions are made in this section for the initial data and weight function $\omega$. 

\medskip
\noindent
{\bf Assumption 1} (weak convergence and being {\em well-prepared})
\begin{align}
\rho_\varepsilon^0\rightharpoonup \rho_0,& \quad\text{in}\quad L^2(\Omega_x) \quad\text{as} \quad\varepsilon\rightarrow 0, \label{weak-cvg1}\\
\lgl \zeta g_\varepsilon^0\rgl\rightharpoonup \lgl \zeta g_0\rgl, &\quad\text{in}\quad L^2(\Omega_x) \quad\text{as} \quad\varepsilon\rightarrow 0, \quad \forall \zeta\in L^2(\Omega_v),
\label{weak-cvg2}\\
\lgl \zeta (g_\varepsilon^0+v\sigma_s^{-1}q_\varepsilon^0) \rgl\rightharpoonup 0, &\quad\text{in}\quad L^2(\Omega_x) \quad\text{as} \quad\varepsilon\rightarrow 0, \quad \forall \zeta\in L^2(\Omega_v). \label{weak-wp}
\end{align}
{\bf  Assumption 2} (boundedness of initial data)
\begin{align}
\sup_{\varepsilon}||\rho^0_\varepsilon||<\infty,\quad\sup_{\varepsilon}|||g_\varepsilon^0|||<\infty,
 \quad\text{and}
 \quad\sup_{\varepsilon}||q_\varepsilon^0||<\infty.
\end{align}
{\bf  Assumption 3} (boundedness for $\omega$)  For any $h$, there exists $\varepsilon_0(h)$, such that 
\beq
2/3 <\omega<2, \quad \forall \eps<\eps_0(h).
\eeq

The assumption for $\omega=\omega(\eps/(\sigma_mh))$ is reasonable due to its property \eqref{w:prop:0}. The next theorem is our main result in terms of the AP property of the IMEX1-LDG method, defined as \eqref{eq:FDG:1T} with \eqref{eq:bilinear:def}-\eqref{eq:init0}.

%%%%%%%%%%%%%%%%%%%%%%%%%%%%%%%
%%  statement of the theorem %%
%%%%%%%%%%%%%%%%%%%%%%%%%%%%%%%
\begin{thm}
\label{thm:AP}
Let the mesh size $h$ be fixed. For any time step size $\Dt$,
there exist unique
$\rho_\LIM^n, u_\LIM^n \in U_h^k$ and $g_\LIM^n\in G_h^k$ for $n\geq 0$, $q_\LIM^n \in U_h^k$ for $n\geq 1$, such that
\begin{subequations}
\label{eq:f:lim}
\begin{align}
&\lim_{\varepsilon\rightarrow 0}W_\AP^n=W_\LIM^n,\quad W=\rho, q, u
\label{eq:aa:-3}\\
&\lim_{\varepsilon\rightarrow 0}\lgl \zeta, g_\AP^n(x,\cdot)\rgl =\lgl \zeta, g_\LIM^n(x, \cdot)\rgl,\quad \forall \zeta\in L^2(\Omega_v),\;\;\forall x\in \Omega_x,
\label{eq:aa:-2}\\
&\lim_{\varepsilon\rightarrow 0}\lgl \zeta, (g_\AP^n, \psi) \rgl =\lgl \zeta, (g_\LIM^n, \psi)\rgl,\quad \forall \zeta\in L^2(\Omega_v), \;\;\forall \psi\in L^2(\Ox).
\label{eq:aa:-1}
\end{align}
\end{subequations}
Furthermore, they satisfy the following scheme
\begin{subequations}
\label{eq:lim:s}
\begin{align}
\label{eq:lim1}
(q_\LIM^{n+1},\testQ)+d_h(\rho_\LIM^{n+1},\testQ)&=0,
\quad\forall\testQ\in U_h^k,\\
\label{eq:lim4}
(\sigma_s u_\LIM^{n+1},\testU)&=(q_\LIM^{n+1},\testU)
\quad\forall\testU\in U_h^k,
\\
\label{eq:lim2}
(\frac{\rho_\LIM^{n+1}-\rho_\LIM^n}{\Dt},\testR)&=\lgl v^2 \rgl  l_h(u_\LIM^{n+1},\testR)-(\sigma_a \rho_\LIM^{n+1},\testR),\quad \forall\testR\in U_h^k,\\
\pi_h(\sigma_s g_\LIM^{n+1})&=- v q_\LIM^{n+1}, \quad  g_\LIM^n+vu^n_\LIM=0,
\label{eq:lim3}
\end{align}
\end{subequations}
for $n\geq 0$, with the initial data $\rho_{\LIM}^0=\pi_h\rho_0$.
This scheme is consistent and stable for the limiting equation \eqref{eq:AP1:3.1.1}, it involves a standard LDG method in space and backward Euler method in time. Therefore the IMEX1-LDG method is AP.
When the velocity space is discrete such as $\Omega_v=\{-1, 1\}$,
 \eqref{eq:aa:-2}-\eqref{eq:aa:-1} can be replaced by a stronger form
 \beq
\lim_{\varepsilon\rightarrow 0}g_\AP^n(\cdot, v)=g_\LIM^n(\cdot, v), \quad \forall v\in \Omega_v.
\label{eq:aa:-4}
\eeq
\end{thm}

\begin{rem}\label{rem:uh}
Alternative to the modal form of the LDG discretization adopted in this work, one can instead consider its nodal form \cite{hesthaven2007nodal}. Most of our analysis in this work can be extended to the resulting nodal methods, with one main difference in how the local equilibrium being satisfied as $\eps\rightarrow 0$. More specifically, using the nodal form, the equations in \eqref{eq:lim:s} containing $\sigma_s$ will be replaced by their nodal counterpart, namely,
$$\sigma_s(x_*)g^{n}_\LIM(x_*, v)=-vq^{n}_\LIM(x_*),
\quad \sigma_s(x_*)u_\LIM^n(x_*)=q_\LIM^n(x_*),$$
where $x_*$ is any nodal point in the discretization. Besides, the absorption terms $\sigma_a \rho$ and $\sigma_a g$ can be treated explicitly in the methods, and interested readers can refer to  \cite{peng2020thesis} for more details on the impact to stability and rigorous AP property. 
\end{rem}

%%%%%%%%%%%%%%%%%%%%%%%%%%%%%%%%%%%%%%%%
% proof for Theorem \ref{thm:mu} 
%%%%%%%%%%%%%%%%%%%%%%%%%%%%%%%%%%%%%%%%
\section{Proof for stability: Theorem \ref{thm:stab_nmu} and Theorem \ref{thm:mu} }
\label{sec:stab:proofs}

In this section, we will  present the proof for Theorem \ref{thm:mu} first and then Theorem \ref{thm:stab_nmu}.

\begin{proof}[\bf{Proof of Theorem  \ref{thm:mu}.}] 
 Let $n\geq 1$. Take $\testR=\rho_h^{n+1}$ in \eqref{eq:FDG:1T:b} and use
Lemma \ref{lem:1} and Proposition \ref{prop:sol}, 
we get
%%%%%%%%%%%%%
\begin{align}
%=
& 
\big(\frac{\rho_h^{n+1}-\rho_h^n}{\Dt}, \rho_h^{n+1}\big)+l_h(\lgl v g_h^n\rgl, \rho_h^{n+1})-\omega \langle v^2\rangle l_h(u_h^{n+1}-u_h^n,\rho_h^{n+1})\notag\\
=& 
\big(\frac{\rho_h^{n+1}-\rho_h^n}{\Dt}, \rho_h^{n+1}\big)
+\lgl  v d_h(  \rho_h^{n+1},  g_h^n)\rgl+\omega \langle v^2\rangle (\sigma_s (u_h^{n+1}-u_h^n),u_h^{n+1}) \notag\\
=& 
\frac{1}{2\Dt}\left(||\rho_h^{n+1}||^2-||\rho_h^n||^2+||\rho_h^{n+1}-\rho_h^n||^2\right)
+\lgl  v d_h(  \rho_h^{n+1},  g_h^n)\rgl\notag\\
&\;\;\;\;\;+
\frac{\omega \lgl v^2\rgl}{2}(
||u_h^{n+1}||_s^2 -||u_h^{n}||_s^2+||u_h^{n+1}-u_h^n||_s^2 )=-(\sigma_a \rho_h^{n+1}, \rho_h^{n+1}).\label{eq:heheda:1}
\end{align}
Take $\testG=\eps^2 g_h^{n+1}$ in \eqref{eq:FDG:1T:d}, integrate over $\Omega_v$ in $v$,  and shift index $n$ to $n-1$, we get
\begin{align}
&\eps^2\lgl\big(\frac{g_h^{n}-g_h^{n-1} }{\Dt},g_h^n\big)\rgl+\eps \lgl b_{h,v}(g_h^{n-1},g_h^n)\rgl-\lgl v d_h(\rho_h^{n},g_h^n)\rgl\notag\\
=&\frac{\eps^2}{2\Dt}\left(|||g_h^n|||^2-|||g_h^{n-1}|||^2+|||g_h^n-g_h^{n-1}|||^2\right)+\eps \lgl b_{h,v}(g_h^{n-1},g_h^n)\rgl-\lgl v d_h(\rho_h^{n},g_h^n)\rgl \notag\\
=&-|||g_h^n|||_s^2 -\eps^2\lgl (\sigma_a g_h^n, g_h^n)\rgl.\label{eq:heheda:2}
\end{align}
Now we sum up \eqref{eq:heheda:1} and \eqref{eq:heheda:2},  with $E_h^n$ defined in \eqref{eq:energy_nmu}, and have
 \begin{align}
&\frac{1}{2\Dt}(E_h^{n+1}-E_h^{n} ) +\frac{1}{2\Dt}(||\rho_h^{n+1}-\rho_h^n||^2+\eps^2|||g_h^n-g_h^{n-1} |||^2)
+\frac{\omega  \lgl v^2\rgl}{2}||u_h^{n+1} -u_h^n||_s^2\notag\\
&+|||g_h^n|||_s^2+ \lgl vd_h(\rho_h^{n+1}-\rho_h^n,g_h^n)\rgl
-\eps\lgl\bh(g_h^n-g_h^{n-1},g_h^n)\rgl+\eps\lgl \bh(g_h^n,g_h^n)\rgl \leq 0. 
\label{eq:process2}
\end{align}

\smallskip
To estimate $\lgl vd_h(\rho_h^{n+1}-\rho_h^n,g_h^n)\rgl$ in \eqref{eq:process2},  
based on the scheme \eqref{eq:FDG:1T:a}-\eqref{eq:FDG:1T:c} and apply the Cauchy-Schwartz inequality, 
we get 
\begin{align}
|\lgl vd_h(\rho_h^{n+1}-\rho_h^n,g_h^n)\rgl|&=| d_h(\rho_h^{n+1}-\rho_h^n,  \lgl vg_h^n \rgl) | =|(q_h^{n+1}-q_h^n,\lgl vg_h^n\rgl) |
\notag\\
&=|(\sigma_s (u_h^{n+1}-u_h^n), \lgl v g_h^n\rgl)| \leq %\sqrt{\lgl v^2\rgl}||q_h^{n+1}-q_h^n||\cdot|||g_h^n|||.
\sqrt{\lgl v^2\rgl} |||g_h^n|||_s\; ||u_h^{n+1}-u_h^n||_s.
%\left(\sqrt {\lgl (g_h^n)^2\rgl},\sigma_s |u_h^{n+1}-u_h^n|\right).
\label{eq:process:dh}
\end{align}
The two terms in \eqref{eq:process2} involving the bilinear form $b_{h,v}$ can be handled  similarly as in \cite{jang2014analysis} (see its Lemma 3.2, particularly equations (3.22)-(3.24)). 
More specifically,  with $\lgl g_h^m\rgl=0$ in Proposition \ref{prop:sol}, utilizing the upwind treatment in the proposed scheme for $v\partial_x g$, in addition to a few applications of inverse inequalities \eqref{eq:inv} and Young's inequality, it can be shown that
\begin{equation}
\lgl \bh(g_h^n,g_h^n)\rgl =\left\lgl\sum_i\frac{|v|}{2}[g_h^n]_\iL^2 \right\rgl,\label{eq:upwind}
\end{equation}
\begin{align}
%&\lgl \bh(g_h^n,g_h^n)\rgl =\left\lgl\sum_i\frac{|v|}{2}[g_h^n]_\iL^2 \right\rgl,\label{eq:upwind}\\
&| \lgl \bh(g_h^{n}-g_h^{n-1} ,g_h^n) \rgl |\notag\\
&\leq (\frac{\theta}{\sigma_m}+\eta) |||g_h^n-g_h^{n-1}|||^2+\frac{\sigma_m}{4\theta}\left\lgl \sum_i \int_{I_i} (v\partial_x g_h^n)^2 dx\right\rgl+\frac{C_{inv} }{4\eta h}\sum_i \left\lgl (v[g_h^n]_\iL)^2\right \rgl\notag\\
&\leq(\frac{\theta}{\sigma_m}+\eta) |||g_h^n-g_h^{n-1}|||^2 +\frac{\widehat{C }_{inv}  ||v^2||_\infty}{4\theta h^2} |||g_h^n|||_s^2
%\left\lgl \left(\sigma_s g_h^n,g_h^n\right)\right\rgl
+\frac{C_{inv} ||v||_{\infty} }{2\eta h}\left\lgl \frac{|v|}{2}\sum_i [g_h^n]_\iL^2  \right\rgl.
\label{eq:process:lambda}
\end {align}
Here $\theta$ and $\eta$ are two  positive constants, which will be specified later.

One important step in this proof is to split $|||g_h^n|||_s^2$
% \lgl(\sigma_s g_h^n,g_h^n)\rgl$ 
in  \eqref{eq:process2} into two terms, each playing different roles,  according to some parameter $\mu\in[0,1]$ (additional conditions required for  $\mu$ will soon become clear),  with one term further rewritten based on the parallelogram identity,
\begin{align}
\label{eq:process:change_g}
|||g_h^n|||_s^2
=\mu|||g_h^n|||_s^2+ (1-\mu)\Big(\frac{1}{2}|||g_h^n|||_s^2-\frac{1}{2}|||g_h^{n-1}|||_s^2
+
\frac{1}{4}|||g_h^n-g_h^{n-1} |||_s^2 + \frac{1}{4}|||g_h^n+g_h^{n-1} |||_s^2\Big).
 \end{align}
We now combine \eqref{eq:process2}-\eqref{eq:process:change_g},  with the discrete energy $E_{h,\mu}^{n}$ defined in \eqref{eqn:energy-mu},
 and reach
\begin{align}
 \label{eq:process:final}
& \frac{1}{2\Dt} (E^{n+1}_{h,\mu}-E^n_{h,\mu})+\eps\left(1-\frac{C_{inv}||v||_\infty}{2\eta h} \right) \left\lgl \frac{|v|}{2}\sum_i[g_h^n]_\iL^2 \right\rgl \\
           &+\left(\frac{\eps^2}{2\Dt}+\frac{1-\mu}{4}\sigma_m-\eps(\frac{\theta}{\sigma_m}+\eta) \right)|||g_h^n-g_h^{n-1} |||^2+(1-\mu)||| \frac{g_h^n+g_h^{n-1} }{2}|||_s^2+\frac{1}{2\Dt}||\rho_h^{n+1}-\rho_h^n||^2 \notag \\
           &+\frac{\omega \lgl v^2\rgl}{2} ||u_h^{n+1}-u_h^n||_s^2-\sqrt{\lgl v^2\rgl} |||g_h^n|||_s\; ||u_h^{n+1}-u_h^n||_s 
          +\left(\mu-\eps\frac{\widehat{C}_{inv}||v^2||_\infty  }
           {4\theta h^2}\right) |||g_h^n|||_s^2 \leq 0.\notag
\end{align}
In order for the discrete energy to be non-increasing, namely,  $ E_{h,\mu}^{n+1}\leq E_{h,\mu}^{n}$, we require the  quadratic form in the final row of \eqref{eq:process:final} to be non-negative, and this can be ensured by a non-negative discriminant, leading to 
\beq
\mu-\eps\frac{\widehat{C}_{inv}||v^2||_\infty }{4\theta h^2}\geq \frac{1}{2\omega}.
\label{eq:ineq:1}
\eeq
Additionally, we also require
\begin{align}
1-\frac{C_{inv} ||v||_\infty}{2\eta h}&\geq 0
 \label{eq:ineq:2} ,\\
\frac{\eps^2}{2\Dt}+\frac{1-\mu}{4}\sigma_m-\eps(\frac{\theta}{\sigma_m}+\eta)&\geq 0
\label{eq:ineq:3}.
\end{align}

The inequality \eqref{eq:ineq:1} implies that  $\mu$ needs to be restricted as   $\mu>\frac{1}{2\omega}$. 
We now choose
$$\frac{\theta}{\sigma_m}=\eta=\frac{1}{2}\left(\frac{\eps}{2\Dt}+\frac{1-\mu}{4\eps}\sigma_m \right), $$
and with this,  \eqref{eq:ineq:3} is satisfied automatically, while \eqref{eq:ineq:2} becomes
\beq
\frac{\eps^2}{\Dt}\geq\frac{4 C_{inv}||v||_\infty\eps-(1-\mu)\sigma_m h }{2 h} ,
\label{eq:result:1}
\eeq
and \eqref{eq:ineq:1} is now 
%leads to
\beq
\frac{\eps^2}{\Dt}\geq
\frac{2\eps^2\widehat{C}_{inv}||v^2||_\infty-(1-\mu)(\mu-\frac{1}{2\omega})\sigma_m^2h^2}{2 (\mu-\frac{1}{2\omega} )\sigma_m h^2}.
\label{eq:result:2}
\eeq
When  $\frac{\eps}{\sigma_m h}\leq\frac{1-\mu}{4C_{inv}||v||_\infty }$, the right hand side of \eqref{eq:result:1} is non-positive, hence \eqref{eq:result:1} holds for any time step $\Dt$. Otherwise, the  time step needs to satisfy $\Dt\leq \tau_{\eps,h,2}(\mu)$ with $\tau_{\eps,h,2}(\mu)$ defined in \eqref{eqn:tau2}. 
Similarly,  when $\frac{\eps}{\sigma_m h}\leq\sqrt{\frac{(1-\mu)(\mu-\frac{1}{2\omega} ) } {2\widehat{C}_{inv}||v^2||_\infty} } $, the right hand side of \eqref{eq:result:2} is non-positive, hence \eqref{eq:result:2} holds for any time step $\Dt>0$. Otherwise,
the time step needs to satisfy $\Dt\leq \tau_{\eps,h,1}(\mu)$ with $\tau_{\eps,h,1}(\mu)$ defined in \eqref{eqn:tau1}. 
The discussions so far can be summarized into  the claims in Theorem \ref{thm:mu} when $k\geq 1$. 

When $k=0$, we have  $\partial_x g_h^{n}=0$,  and the estimate in \eqref {eq:process:lambda} can be replaced by
\beq
| \lgl \bh(g_h^{n}-g_h^{n-1} ,g_h^n) \rgl \leq\eta |||g_h^n-g_h^{n-1}|||^2+\frac{C_{inv} ||v||_{\infty} }{2\eta h}\left\lgl \frac{|v|}{2}\sum_i [g_h^n]_\iL^2  \right\rgl,
\label{eq:process:k0}
\eeq
and all analysis up to \eqref{eq:ineq:3}  holds without the terms containing $\theta$. Specifically, \eqref{eq:ineq:1}-\eqref{eq:ineq:3} become
\beq
\mu \geq \frac{1}{2\omega}, \quad
1-\frac{C_{inv} ||v||_\infty}{2\eta h}\geq 0,\quad
\frac{\eps^2}{2\Dt}+\frac{1-\mu}{4}\sigma_m -\eps\eta\geq 0
\label{eq:ineq:4-6}.
\eeq
Now take
$$
\eta= \frac{\eps}{2\Dt}+\frac{1-\mu}{4\eps}\sigma_m
$$
in \eqref{eq:ineq:4-6}, and follow a similar analysis as  above, one reaches the results for $k=0$.
\end{proof}

\begin{proof}[\bf{Proof of Theorem  \ref{thm:stab_nmu}.}]   The proof can be established by starting with  the equation \eqref{eq:process2}, and then following almost the identical  analysis  in \cite{jang2014analysis} (particularly, see equations (3.22), (3.26)-(3.28), (3.36)-(3.41) in \cite{jang2014analysis}), together with $|||g_h^n|||_s^2\geq \sigma_m |||g_h^n|||^2$ to deal with the general scattering coefficient $\sigma_s(x)$. The details are omitted.
\end{proof}

%\subsection{Optimization of the stability condition by choosing $\mu$}
\section{Proof for stability: Theorem \ref{thm:stab} }
\label{sec:stab:proofs:opt}

When $k=0$, the optimization is straightforward, and the detail is omitted. The remaining of this section will be devoted to the case when $k\geq 1$, for which the analysis is more technically involved.  
From here on, we assume $1\leq k\leq 9$. With this,  we have $\mK>1$ and $\widehat{C}_{inv}>0$. We also assume $\omega>1/2$, though not all preliminary results next depend on this assumption. One can refer to Table \ref{tab:notation} for a summary of notation.

\subsection{Preliminary lemmas}

We first state and prove some preparatory lemmas.  Lemma \ref{lem:muS:mono} and Lemma \ref{lem:tau12} can be directly verified and the proofs are skipped. 
%%%%%%%%%%%%%%%%%
% lemma
%%%%%%%%%%%%%%%%%

\begin{lem} 
\label{lem:muS:mono}
\begin{itemize}
\item[(i)]  With $\omega>1/2$, there always holds   $\mu_\star\in (\frac{1}{2\omega}, 1)$.
\item[(ii)]  With $\mu_S(\lambda)$ defined
 in \eqref{eq:thm:stab5:o:3}, let its inverse be  $\lambda_S(\mu): =2(\mu-\frac{1}{2\omega})\frac{\bbb}{\aaa}$. 
\begin{itemize}
\item Both $\mu_S(\lambda)$ and $\lambda_S(\mu)$ are monotonically increasing. And $\mu_S(\lambda)>\frac{1}{2\omega}, \forall \lambda>0$.
\item With $\widehat{\lambda}_\star=\lambda_S(1)$, we have $\mu_S(\widehat{\lambda}_\star) = 1$. In addition, $ \mu_S(\lambda)<1\Leftrightarrow \lambda <\widehat{\lambda}_\star$. 
\item  $\mu_S(\lambda_\star)= \mu_\star$ and $\lambda_S(\mu_\star)=\lambda_\star$.
\end{itemize}
\end{itemize}
\end{lem}

%%%%%
% lemma
%%%%%
\begin{lem}\label{lem:mu_star}
Consider $\mu\in (\frac{1}{2\omega}, 1]$, then 
\begin{itemize}
\item[(i)] 
\begin{equation}
\lamone(\mu)  \leq \lamtwo(\mu)\Longleftrightarrow \mu\leq \mu_\star \left(\Longleftrightarrow \frac{1}{2\omega}<\mu\leq \mu_\star<1\right),
\end{equation}
and $\lamone(\mu_\star) = \lamtwo(\mu_\star)=\lambda_\star$. In addition, $\lamone(\mu)$ is monotonically increasing on
$(\frac{1}{2\omega}, \mu_\star]$, and   $\lamtwo(\mu)$ is monotonically decreasing.
\item[(ii)] 
\begin{equation}
\lambda_S(\mu) \leq \lamone(\mu)\Longleftrightarrow \mu\leq \mu_\star \left(\Longleftrightarrow \frac{1}{2\omega}<\mu\leq \mu_\star<1\right).
\end{equation}
\item[(iii)] 
\beq
\widehat{\lambda}_\star>\lamone(\mu), \quad \widehat{\lambda}_\star> \lamtwo(\mu), \quad \forall \mu\in (\frac{1}{2\omega}, 1].
\eeq
\end{itemize}
\end{lem}
\begin{proof} For $\mu\in (\frac{1}{2\omega}, 1]$,  to prove (i), 
\begin{align*}
\lamone(\mu)  \leq \lamtwo(\mu)&\Longleftrightarrow\sqrt{\frac{(1-\mu)(\mu-\frac{1}{2\omega}) }{2\widehat{C}_{inv} ||v^2||_\infty} } \leq \frac{1-\mu}{4C_{inv}||v||_\infty}\\
&\Longleftrightarrow
\frac{\mu-\frac{1}{2\omega} }{\aaa} \leq \frac{1-\mu}{8(\bbb)^2}
\Longleftrightarrow  \mu  \leq  \mu_\star.
\end{align*}
The equality is achieved at $\mu=\mu_\star$, with the value being  $\lambda_\star$.  The monotonicity of $\lamtwo(\mu)$ is straightforward.
For $\lamone(\mu)$, note that with $\mK>1$, we have
 $\mu_\star<\frac{1}{2}\left(1+\frac{1}{2\omega}\right)$, with $\frac{1}{2}\left(1+\frac{1}{2\omega}\right)$ being  where  $\lamone(\mu)$ achieves its maximum. This implies that $\lamone(\mu)$, whose square is a downward-facing parabola,  is monotonically increasing on $(\frac{1}{2\omega}, \mu_\star]$.

To prove (ii),  we proceed as below.
\begin{align*}
\lambda_S(\mu) \leq \lamone(\mu) \Longleftrightarrow &2(\mu-\frac{1}{2\omega})\frac{\bbb}{\aaa}
\leq 
\sqrt{\frac{(1-\mu)(\mu-\frac{1}{2\omega} ) } {2\widehat{C}_{inv} ||v^2||_\infty} }
\notag\\
\Longleftrightarrow &  (\mu-\frac{1}{2\omega})\frac{8(\bbb)^2}{\aaa} \leq 1-\mu 
\Longleftrightarrow  \mu \leq \mu_\star.
\end{align*}

To prove (iii), related to  $\lamtwo(\mu)$, given its being  monotonically decreasing, we only need to show 
$\widehat{\lambda}_\star>\lamtwo(\frac{1}{2\omega})$, which is ensured by $\mK>1$ as below.
\beq
\widehat{\lambda}_\star>\lamtwo(\frac{1}{2\omega})\Longleftrightarrow
2(1-\frac{1}{2\omega})\frac{\bbb}{\aaa}>\frac{1-\frac{1}{2\omega}}{4\bbb}
\Longleftrightarrow \mK>1.
\eeq
Related to $\lamone(\mu)$, from the proof of (i) of this lemma, we only need to verify 
$\widehat{\lambda}_\star>\lamone(\mu)|_{\mu= \frac{1}{2}(1+\frac{1}{2\omega})}.$ This can be argued as follows.
\beq
\widehat{\lambda}_\star>\lamone(\mu)|_{\mu= \frac{1}{2}(1+\frac{1}{2\omega})}\Longleftrightarrow
2(1-\frac{1}{2\omega})\frac{\bbb}{\aaa}>\frac{1-\frac{1}{2\omega}}{2\sqrt{2\aaa }}
\Longleftrightarrow 4\mK>1.
\eeq
This holds due to that $\mK>1$. 
\end{proof}

%%%%%%%%%%%%%%%%%%%%%%%%%%%%%%%%%%%%%%%%%%%%%%%%%%%%%%%%%%%%%%%%%%%%%%%%%%%%%
\begin{rem}\label{rem:geometry}
Lemmas  \ref{lem:muS:mono}-\ref{lem:mu_star} tell the properties  and the relative locations of the curves $\lambda=\lambda_S(\mu)$, $\lambda=\lamone(\mu)$ and $\lambda=\lamtwo(\mu)$. Particularly, 
\begin{itemize}
\item
According to Lemmas  \ref{lem:muS:mono}-\ref{lem:mu_star},   the curves $\lambda=\lambda_S(\mu)$, $\lambda=\lamone(\mu)$ and $\lambda=\lamtwo(\mu)$ intersect at $(\mu_\star,\lambda_\star)$.
\item 
According to Lemma \ref{lem:mu_star}, to the left of  $\mu=\mu_\star$, the graph of $\lambda=\lamtwo(\mu)$ is above  that of $\lambda=\lamone(\mu)$, which is above the graph of $\lambda=\lambda_S(\mu)$; to the right of $\mu=\mu_\star$,  the ordering is reversed. 
\end{itemize}
\end{rem}

It is important to know the relative locations of various curves to optimize the time step condition. For general weight function $\omega$, it is nontrivial to visualize these curves, yet their  relative locations and some special points are captured in Figure \ref{fig:proof}, which is for the constant weight function $\omega\equiv 1$.  The figure can also facilitate the readers to follow and understand the analysis in this section, which is given algebraically  for general $\omega$ and has a geometric interpretation  for the special case of $\omega\equiv 1$.

\begin{figure}[!h]
  \centering
  {\includegraphics[width=0.5\textwidth]{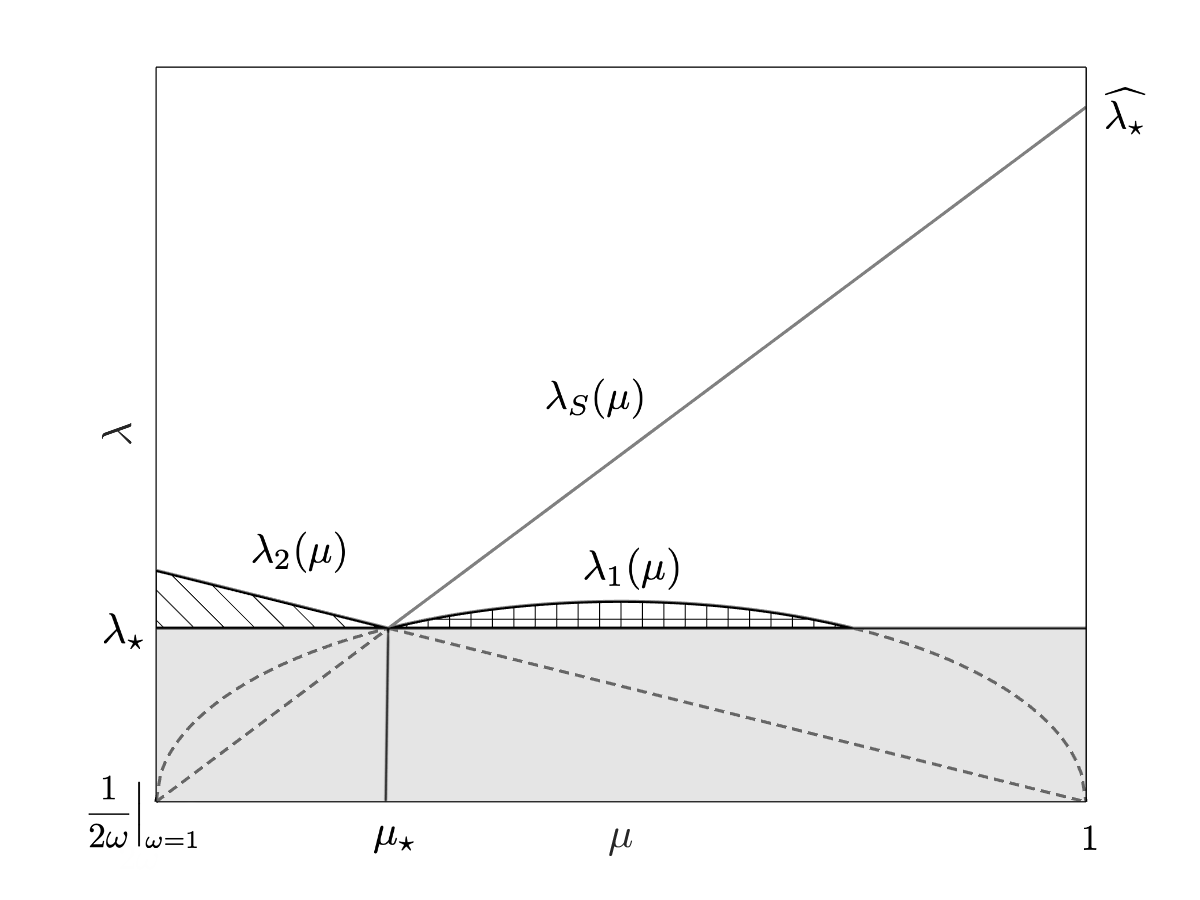}}
  \caption{ Plots with constant $\omega\equiv1$ to facilitate the understanding of  Lemmas \ref{lem:muS:mono}-\ref{lem:mu_star}. The scheme is: i) unconditionally stable when $\lambda=\eps/(\sigma_m h)$ and $\mu$ fall into the gray region, ii) $\mu$-stable under $\Dt\leq \tau_{\eps,h,1}(\mu)$ in the stripped region,  iii) $\mu$-stable under $\Dt\leq \tau_{\eps,h,2}(\mu)$ in the latticed region, and iv) $\mu$-stable under $\Dt\leq \min(\tau_{\eps,h,1}(\mu), \tau_{\eps,h,2}(\mu))$ in the  blank (white) region.}
  \label{fig:proof}
  \end{figure}

%%%%%%%%%%%%%%%%%%%%%%%%%%%%%%
% lemma: tau1(mu_s) and tau2(mu_s) are finite
%%%%%%%%%%%%%%%%%%%%%%%%%%%%%%
\begin{lem}\label{lem:tau12}
When $\frac{\eps}{\sigma_m h}> \max (\lamone(\mu), \lamtwo(\mu))$, both %$\tau_{\eps, h, 1}(\mu)$ and $\tau_{\eps, h, 2}(\mu)$ 
$\tauone(\mu)$ and $\tautwo(\mu)$ are finite, and they satisfy 
\begin{equation}
\tauone(\mu)=\tau_{\eps, h, 1}(\mu)\leq \tautwo(\mu)=\tau_{\eps, h, 2}(\mu)\Longleftrightarrow \mu\leq \mu_S(\frac{\eps}{\sigma_m h}) \Longleftrightarrow 
\lambda_S(\mu)\leq \frac{\eps}{\sigma_m h}.
\end{equation}
Moreover, $\tau_{\eps,h,1}(\mu_S(\frac{\eps}{\sigma_m h}) )= \tau_{\eps,h,2}(\mu_S(\frac{\eps}{\sigma_m h}) ).$
\end{lem}

%%%%%%%%%%%%%%%%%%%%%%%%%%%%%%%%%%%
% lemma: monotonicity of tau1 and tau2
%%%%%%%%%%%%%%%%%%%%%%%%%%%%%%%%%%%
\begin{lem}\label{lem:monotone}
When restricted to $\{\mu: \frac{\eps}{\sigma_m h}>\lamtwo(\mu)\}$,  $\tau_{\eps,h,2}(\mu)$ is positive and  monotonically decreasing. 
When restricted to $\{\mu\in (\frac{1}{2\omega}, \min(\mu_S(\frac{\eps}{\sigma_m h}), 1)]: \frac{\eps}{\sigma_m h}>\lamone(\mu) \}$,  $\tau_{\eps,h,1}(\mu)$ is positive and monotonically increasing. 
\end{lem}
%%%%
%proof
%%%
\begin{proof}
The definitions of $\lambda_{j}(\mu)$ ensures  $\tau_{\eps,h,j}(\mu)$  is positive with $j=1, 2$ for the considered $\mu$. The monotonicity of $\tau_{\eps,h,2}(\mu)$ directly comes from its being linear, and what remained will be devoted to showing the monotonicity of  $\tau_{\eps,h,1}(\mu)$. 

Based on the definition of $\tau_{\eps,h,1}(\mu)$ in \eqref{eqn:tau1},  we know that when $\frac{\eps}{\sigma_m h}>\lamone(\mu)$, we have $2\eps^2 \widehat{C}_{inv}||v^2||_\infty-(1-\mu)(\mu-\frac{1}{2\omega} )\sigma_m^2h^2>0$, and 
\begin{align*}
\tau'_{\eps,h,1}(\mu)=\frac{2\eps^2h^2\sigma_m\left(2\eps^2 \widehat{C}_{inv}|| v^2||_\infty-(\mu-\frac{1}{2\omega})^2\sigma_m^2h^2\right)
}{(2\eps^2 \widehat{C}_{inv}||v^2||_\infty-(1-\mu)(\mu-\frac{1}{2\omega} )\sigma_m^2h^2)^2}.
\end{align*}
As a result, the sign of $\tau'_{\eps,h,1}(\mu)$,  same as that of
$
q(\mu):=
2\eps^2 \widehat{C}_{inv}|| v^2||_\infty-(\mu-\frac{1}{2\omega})^2\sigma_m^2h^2,$
will inform  about  the monotonicity of $\tau_{\eps,h,1}(\mu)$.

Consider the two roots of $q(\mu)$, which are  $\tilde{\mu}_{1,2}=\tilde{\mu}_{1,2}(\frac{\eps}{\sigma_m h})=\frac{1}{2\omega}\mp \dfrac{\eps}{\sigma_m h}\sqrt{2\widehat{C}_{inv}||v^2||_\infty }.$ And $q(\mu)>0$ when  $\mu \in(\tilde{\mu}_1,\tilde{\mu}_2)$.  Note that $\tilde{\mu}_1<\frac{1}{2\omega}$. One can further show that $\tilde{\mu}_2(\lambda)> \mu_S(\lambda), \; \forall \lambda>0$ as below.
\begin{align*}
\mu_S(\lambda)<\tilde{\mu}_2(\lambda)&\Longleftrightarrow\frac{1}{2\omega}+\frac{1}{2}\lambda\frac{\widehat{C}_{inv}||v^2||_\infty }{C_{inv}||v||_\infty}< \frac{1}{2\omega}+\lambda\sqrt{2\aaa}
\notag\\
&\Longleftrightarrow \frac{\widehat{C}_{inv}||v^2||_\infty }{2C_{inv}||v||_\infty}<\sqrt{2\widehat{C}_{inv}||v^2||_\infty }\Longleftrightarrow \mK> 1.
\end{align*}
Hence $\left(\frac{1}{2\omega}, \min(\mu_S(\frac{\eps}{\sigma_m h}), 1)\right]\subset (\tilde{\mu}_1,\tilde{\mu}_2)$. And the monotonicity of $\tau_{\eps,h,1}(\mu)$ will follow.

\end{proof}

\begin{lem}\label{lem:lambda_star}
Assume $\lambda>0$. 
\begin{itemize}
\item[(i)] $\lambda>\lambda_\star \Longleftrightarrow \lambda >\lamtwo(\mu_S(\lambda))$.%= \frac{1-\mu_S(\lambda)}{4\bbb}$;
\item[(ii)] When $\lambda\leq \widehat{\lambda}_\star$,  then $\lambda>\lambda_\star \Longleftrightarrow  \lambda>\lamone(\mu_S(\lambda))$.
\item[(iii)] When $\lambda_\star<\frac{\eps}{\sigma_m h}\leq  \widehat{\lambda}_\star$,   we have $\frac{\eps}{\sigma_m h}> \max (\lamone(\mu), \lamtwo(\mu))|_{\mu=\mu_S(\frac{\eps}{\sigma_m h})}$.
\end{itemize}
\end{lem}
\begin{proof}
To prove (i), we proceed from the definitions of  $\lamtwo(\mu)$ and $\mu_S(\lambda)$, and get
\begin{align}
\lambda >\lamtwo(\mu_S(\lambda)) \Longleftrightarrow & \lambda >\frac{1-\frac{1}{2\omega}-\frac{1}{2}\lambda \frac{\aaa}{\bbb} }{4\bbb} \\
\Longleftrightarrow &
\left(1+ \frac{\aaa}{8(\bbb)^2} \right)\lambda > \frac{ 1-\frac{1}{2\omega} }{4\bbb}
\Longleftrightarrow\lambda>  \lambda_\star.
\end{align}
To prove (ii),  we first notice  $\mu_S(\lambda)>\frac{1}{2\omega}$ holds  when $\lambda>0$. With $\lambda\leq \widehat{\lambda}_\star$, equivalently $\mu_S(\lambda)\leq 1$,  we then have
\begin{align}
\lambda>\lamone(\mu_S(\lambda)) \Longleftrightarrow & \lambda > \sqrt{\frac{(1-\frac{1}{2\omega}-\frac{1}{2}\lambda \frac{\aaa}{\bbb})\frac{1}{2}\lambda \frac{\aaa}{\bbb} }{2\aaa} } \notag\\
\Longleftrightarrow&
\lambda > \left(1-\frac{1}{2\omega} - \frac{1}{2} \lambda \frac{\aaa}{\bbb} \right) \frac{1}{4\bbb}
\Longleftrightarrow \lambda> \lambda_\star.
\end{align}
(iii) is a direct result of (i) and (ii) of this lemma.
\end{proof}

%%%%%%%%%%%%%%%%%%%%%%%%%%%%%%%%%%%%%%%%%%%%%%%%%%%%%%%

\subsection{Proof of Theorem \ref{thm:stab}: unconditionally stable region, $k\geq 1$}\label{sec:un_k1}

Based on Theorem \ref{thm:mu} and the definition of (unconditional) stability, the IMEX1-LDG method is unconditionally stable if and only if $\dtstabk(\eps, h)=\infty$, which is equivalent to 
\begin{align}
\label{eq:uncon-pk}
\frac{\eps}{\sigma_m h}\leq  \max_{\mu\in(\frac{1}{2\omega},1]}\big(\min\big(\lamone(\mu),\lamtwo(\mu) \big)\big).
%=\max_{\mu\in(\frac{1}{2\omega},1]}\left\{ \min\left(\frac{1-\mu}{4C_{inv}||v||_\infty }, \sqrt{\frac{(1-\mu)(\mu-\frac{1}{2\omega} ) } {2\widehat{C}_{inv} ||v^2||_\infty} } \right)\right\}.
\end{align}
Using Lemma \ref{lem:muS:mono}-(i) and  Lemma \ref{lem:mu_star}-(i), one has
\beq
 \min\big(\lamone(\mu),\lamtwo(\mu)\big)
=\left\{
\begin{array}{ll}
\lamone(\mu), & \textrm{if}\; \mu\leq \mu_\star, \\
\lamtwo(\mu), &\textrm{if}\; \mu\geq \mu_\star
\end{array}
\right.
\eeq
where $\mu_\star\in (\frac{1}{2\omega},  1)$, and  the inequality \eqref{eq:uncon-pk} will be simplified as
  \begin{align}
\frac{\eps}{\sigma_m h}\leq 
\max\left(
 \max_{\mu\in(\frac{1}{2\omega},\mu_\star]} \lamone(\mu), 
\max_{\mu\in[\mu_\star,1]}\lamtwo(\mu)
\right)
=\max\left(\lamone(\mu_\star), \lamtwo(\mu_\star)\right)=\lambda_\star.
\label{eq:uncon-pk:2}
\end{align}
This gives the result in Theorem \ref{thm:stab} regarding the unconditional stability  when  $k\geq 1$.

%%%%%%%%%%%%%%%%%%%%%%%%%%%%%%%%%%%%%%%%%%%%%%%%%%%%%%%%%%%%%%%
\subsection{Proof of Theorem \ref{thm:stab}:  conditionally stable region, $1\leq k\leq 9$, $ \frac{\eps}{\sigma_m h}>\lambda_\star$}
 %%%%%%%%%%%%%%%%%%%%% %%%%%%%%%%%%%%%%%%%%%

In this subsection, we focus on $\eps$ and $h$ that satisfy $ \frac{\eps}{\sigma_m h}>\lambda_\star$. For such $\eps, h$, we have $\Dt_{\textrm{stab}}(\eps, h)<\infty$, and the IMEX1-LDG method is conditionally stable. Based on the $\mu$-stability result in Theorem \ref{thm:mu}, we want to optimize the time step condition by properly choosing $\mu$ from the admissible set, hence to get $\dtstabk(\eps, h)$ and establish the remaining result in Theorem \ref{thm:stab}.  %Throughout this subsection  section,  $ \frac{\eps}{\sigma_m h}>\lambda_\star$ is assumed.

\subsubsection{When $ \frac{\eps}{\sigma_m h}>\widehat{\lambda}_\star$}
We start with the simplest case, that is when  $ \frac{\eps}{\sigma_m h}>\widehat{\lambda}_\star$. According to Lemma \ref{lem:mu_star}-(iii),  for such $\eps, h$, one has $\frac{\eps}{\sigma_m h}>\max(\lamone(\mu),\lamtwo(\mu)), \forall \mu\in(\frac{1}{2\omega}, 1]$, hence $\tau_{\eps, h, j}(\mu)<\infty$, $j=1, 2$, and $$\Dt_{\textrm{stab}}(\eps, h)=\max_{ \mu\in(\frac{1}{2\omega}, 1]}\min (\tau_{\eps, h, 1}(\mu), \tau_{\eps, h, 2}(\mu)).$$ 

Using the property of $\mu_S(\lambda)$ in Lemma \ref{lem:muS:mono},  we get
\beq
\frac{\eps}{\sigma_m h}>\widehat{\lambda}_\star\Leftrightarrow \mu_S(\frac{\eps}{\sigma_m h})>\mu_S(\widehat{\lambda}_\star)=1\Rightarrow \mu_S(\frac{\eps}{\sigma_m h}) \geq \mu, \quad \forall \mu\in (\frac{1}{2\omega}, 1].
\label{eq:Dt0}
\eeq
Now following the comparison property in Lemma \ref{lem:tau12} and the monotonicity of $\tau_{\eps, h, 1}(\mu)$ in Lemma \ref{lem:monotone}, we have, when $\frac{\eps}{\sigma_m h}>\widehat{\lambda}_\star$,
$$\Dt_{\textrm{stab}}(\eps, h)=\max_{ \mu\in(\frac{1}{2\omega}, 1]\cap(\frac{1}{2\omega},\mu_S(\frac{\eps}{\sigma_m h})] }\tau_{\eps, h, 1}(\mu)=
\tau_{\eps, h, 1}\left( \min(\mu_S(\frac{\eps}{\sigma_m h}),   1) \right).$$

 \subsubsection{When $ {\lambda}_\star<\frac{\eps}{\sigma_m h}\leq \widehat{\lambda}_\star$} %
 %%%%%%%%%%%%%%%%%%%%% %%%%%%%%%%%%%%%%%%%%%
From here on, we assume $ \frac{\eps}{\sigma_m h}\in (\lambda_\star, \widehat{\lambda}_\star]$. 
The relation in \eqref{eq:Dt0} implies  
\beq
\mu_S(\frac{\eps}{\sigma_m h})\leq 1.
\label{eq:Dt4}
\eeq  We decompose  $(\frac{1}{2\omega}, 1]$ into three disjoint sets $S_j(\eps, h)$, $j=1, 2, 3$, defined as
\begin{align*}
%\label{eqn:A1}
S_1(\eps,h) &= \left\{\mu\in  (\frac{1}{2\omega}, 1]: \frac{\eps}{\sigma_m h}>\max(\lamone(\mu),\lamtwo(\mu))\right\},\\
S_2(\eps,h) &= \left\{\mu\in  (\frac{1}{2\omega}, 1]: \lamone(\mu)<\frac{\eps}{\sigma_m h}\leq \lamtwo(\mu)\right\},\\
S_3(\eps,h) &= \left\{\mu\in  (\frac{1}{2\omega}, 1]: \lamtwo(\mu)<\frac{\eps}{\sigma_m h}\leq \lamone(\mu))\right\}.
\end{align*}
One can refer to Figure \ref{fig:proof} to visualize the decomposition for a constant weight function $\omega\equiv 1$.
And correspondingly, 
$$\Dt_{\textrm{stab}}(\eps, h)=\max_{\mu\in (\frac{1}{2\omega}, 1]}\min(\tauone(\mu), \tautwo(\mu))=\max_{j=1, 2, 3} \Dt_{\textrm{stab}}^{(j)}(\eps, h),$$
where  
$\Dt_{\textrm{stab}}^{(j)}(\eps, h):=\max_{\mu\in S_j(\eps, h)}\min(\tauone(\mu), \tautwo(\mu)).$ 
Next we will calculate  $\Dt_{\textrm{stab}}^{(1)}(\eps, h)$, and then show
$\Dt_{\textrm{stab}}^{(1)}(\eps, h)\geq \Dt_{\textrm{stab}}^{(j)}(\eps, h), j=2, 3$, therefore
\beq \Dt_{\textrm{stab}}(\eps, h)= \Dt_{\textrm{stab}}^{(1)}(\eps, h).\eeq

\bigskip
 {\bf Step 1:  To compute $\Dt_{\textrm{stab}}^{(1)}(\eps, h)$.}  When $\mu\in S_1(\eps, h)$, we have $\tauone(\mu)=\tau_{\eps,h,1}(\mu)<\infty$, $\tautwo(\mu)=\tau_{\eps,h,2}(\mu)<\infty$.
 Based on the comparison result in  Lemma \ref{lem:tau12},  and the property of $\mu_S(\lambda)$ in Lemma \ref{lem:muS:mono}, there holds
 
\begin{align}
\label{eq: mtau_12}
\min\left(\tauone(\mu),\tautwo(\mu) \right)
&=\begin{cases}
\tau_{\eps,h,1}(\mu), \quad & \mu \in (\frac{1}{2\omega},\mu_S(\frac{\eps}{\sigma_m h}) ],\\
\tau_{\eps,h,2}(\mu), \quad & \mu\in  (\mu_S(\frac{\eps}{\sigma_m h}),1].
\end{cases}
\end{align}
 
With  $ \lambda_\star<\frac{\eps}{\sigma_m h}\leq \widehat{\lambda}_\star$, based on 
Lemma \ref{lem:lambda_star}-(iii), we will get $\mu_S(\frac{\eps}{\sigma_m h})\in S_1(\eps, h)$.    
By further using the monotonicity of $\tau_{\eps, h, j}(\mu), j=1, 2$ in Lemma \ref{lem:monotone}, and the fact $\tau_{\eps,h,1}(\mu_S(\frac{\eps}{\sigma_m h}) ) = \tau_{\eps,h,2}(\mu_S(\frac{\eps}{\sigma_m h}) )$ in Lemma \ref{lem:tau12},  when  $ \frac{\eps}{\sigma_m h}\in (\lambda_\star, \widehat{\lambda}_\star]$,
\begin{align}
\label{eq:Dt1}
\Dt_{\textrm{stab}}^{(1)}(\eps, h) &= \max_{\mu \in S_1(\eps,h)}\Big( \min\left(\tauone(\mu),\tautwo(\mu)\right)\Big)\notag\\
& =
\tau_{\eps,h,1}( \mu_S(\frac{\eps}{\sigma_m h}) ) 
= \tau_{\eps,h,1} \left( \min(\mu_S(\frac{\eps}{\sigma_m h}), 1) \right).
\end{align}

%%%%%%%%%%%%%%%%%%%%%%%%%%%%%%%%%%%%%%%%%%%%%%%%%%%%%%%%%%%%%%%%%%%
\bigskip
 {\bf Step 2:  To show  $\Dt_{\textrm{stab}}^{(2)}(\eps, h)\leq \Dt_{\textrm{stab}}^{(1)}(\eps, h)$.}  When $\mu\in S_2(\eps, h)$, we have $\tauone(\mu)=\tau_{\eps,h,1}(\mu)<\infty$, $\tautwo(\mu)=\infty$, hence $\min(\tauone(\mu),\tautwo(\mu) ) = \tau_{\eps,h,1}(\mu)$.

For any $\mu\in S_2(\eps, h)$, based on Lemma \ref{lem:mu_star}, we have $\mu \leq \mu_\star$. Moreover, using the fact of 
$\mu_S(\lambda_\star) = \mu_\star$ and the monotonicity of $\mu_S(\lambda)$ in Lemma \ref{lem:muS:mono}, as well as the assumption $\frac{\eps}{\sigma_m h}>\lambda_\star$, we have for $\mu\in S_2(\eps, h)$,
$$
\mu \leq \mu_\star=\mu_S(\lambda_\star)< \mu_S(\frac{\eps}{\sigma_m h}).
$$
Finally, we can once again use  the monotonicity of  $\tau_{\eps,h,1}(\mu)$ in Lemma \ref{lem:monotone}, and conclude 
\begin{align}
\label{eq:Dt2}
\Dt_{\textrm{stab}}^{(2)}(\eps, h) &= \max_{\mu \in S_2(\eps,h)}\Big( \min\left(\tauone(\mu),\tautwo(\mu)\right)\Big)= \max_{\mu \in S_2(\eps,h)}\tau_{\eps, h, 1}(\mu)\notag\\
&\leq\tau_{\eps,h,1}\left(\mu_S(\frac{\eps}{\sigma_m h})\right) = \Dt_{\textrm{stab} }^{(1)}(\eps, h).
\end{align}

%%%%%%%%%%%%%%%%%%%%%%%%%%%%%%%%%%%%%%%%%%%%%%%%%%%%%%%%%%%%%%%%%%%
\bigskip
 {\bf Step 3:  To show  $\Dt_{\textrm{stab}}^{(3)}(\eps, h)\leq \Dt_{\textrm{stab}}^{(1)}(\eps, h)$.}   When $\mu\in S_3(\eps, h)$, we have $\tauone(\mu)=\infty$, $\tautwo(\mu)=\tau_{\eps,h,2}(\mu)<\infty$, hence $\min(\tauone(\mu),\tautwo(\mu) ) = \tau_{\eps,h,2}(\mu)$.

Given any  $\mu \in S_3(\eps, h)$, we know $\lamtwo(\mu)<\frac{\eps}{\sigma_m h}\leq \lamone(\mu)$. This, combined with Lemma \ref{lem:mu_star}, implies $\mu>\mu_\star$, and additionally
\begin{align}
\frac{\eps}{\sigma_m h} \leq \lamone(\mu)< \lambda_S(\mu) \Leftrightarrow \mu> \mu_S\left(\frac{\eps}{\sigma_m h}\right). 
\end{align}
The equivalency is based on the monotonicity of $\mu_S(\lambda)$ in Lemma \ref{lem:muS:mono}. Finally one can use the 
 monotonicity of $\tau_{\eps,h,2}(\mu)$ in Lemma \ref{lem:monotone}, and conclude
\begin{align}
\Dt_{\textrm{stab}}^{(3)}(\eps, h) &=
\max_{\mu \in S_3(\eps,h)}\Big( \min\left(\tauone(\mu),\tautwo(\mu)\right)\Big)= \max_{\mu \in S_3(\eps,h)}\tau_{\eps, h, 2}(\mu)\notag\\
& \leq \tau_{\eps,h,2}(\mu_S(\frac{\eps}{\sigma_m h}) )=\tau_{\eps,h,1}\left(\mu_S(\frac{\eps}{\sigma_m h}) \right)
=\Dt_\textrm{stab}^{(1)}(\eps, h).\notag
\end{align}

%%%%%%%%%%%%%%%%%%%%%%%%%%%%%%%%%%%%%%%%%%%%%%%%%%%%%%%%%%%%%%%%%%%

\section{Proof for AP property: Theorem \ref{thm:AP}}
\label{sec:ap:proofs}

We will first build some preparatory results in Lemma \ref{lem:AP}, before proving the main result on the AP property in Theorem \ref{thm:AP}. The three assumptions in Section \ref{sec:ap} still hold.  Let $\{\Psi_j\}_{j=1}^{N_k}$ be an orthonormal basis of $U_h^k$ with respect to the standard $L^2$ inner product of $L^2(\Omega_x)$. Recall the initialization is via the $L^2$ projection onto $U_h^k$, namely, $\rho^0_\AP=\pi_h \rho_\varepsilon^0$, $g^0_\AP=\pi_h g_\varepsilon^0$,  $u_\AP^0=\pi_h(\sigma_s^{-1}q^0_\eps)$. We also define $W_{\LIM}^0=\pi_h W_0$ for $W=\rho, g$, and $u_\LIM^0=\pi_h(\sigma_s^{-1}q_0)$.

%%%%%%%%%%%%%%%%
% Lem
%%%%%%%%%%%%%%%%
\begin{lem}\label{lem:AP} The following results hold. \\
(i) $q_\varepsilon^0\rightharpoonup q_0$  in $L^2(\Omega_x)$ as $\varepsilon\rightarrow 0$. \\
(ii) $\lim_{\varepsilon\rightarrow 0}\rho_\AP^0=\rho^0_{\LIM}$, 
$\lim_{\varepsilon\rightarrow 0}u_\AP^0=u_{\LIM}^0$ and
\begin{align}
&\lim_{\varepsilon\rightarrow 0}\lgl \zeta, g_\AP^0(x,\cdot)\rgl =\lgl  \zeta, g_\LIM^0(x, \cdot)\rgl,\quad \forall  \zeta\in L^2(\Omega_v),\;\;\forall x\in \Omega_x,\\
\label{eq:g_initial_limit2}
&\lim_{\varepsilon\rightarrow 0}\lgl  \zeta, (g_\AP^0, \psi) \rgl =\lgl  \zeta, (g_\LIM^0, \psi)\rgl,\quad \forall  \zeta\in L^2(\Omega_v), \;\;\forall \psi\in L^2(\Ox).
\end{align}
(iii) $\sup_{\varepsilon}||W_\AP^0||<\infty$, where $W=\rho, g, u$.\\
(iv) $\sup_{\{0<\varepsilon<\eps_0(h)\}}||W_\AP^1||=C_W(k,\Dt, h,\Ov)<\infty$, where $W=\rho,  u$.
\end{lem}
\begin{proof}[\bf{Proof}]
\par

 (i) Start with any $\phi\in C^\infty_0(\Omega_x)$, then
 \beq
 (q_0, \phi)= - (\rho_0, \phi_x)=-\lim_{\varepsilon\rightarrow0}(\rho_\varepsilon^0, \phi_x)=\lim_{\varepsilon\rightarrow0} (q_\varepsilon^0, \phi).
\eeq
 This result can be extended to any $\phi\in L^2(\Omega_x)$, hence $q_\varepsilon^0\rightharpoonup q_0$  in $L^2(\Omega_x)$ as $\eps\rightarrow 0$, due to the uniform boundedness of $||q_\varepsilon^0||$ in $\varepsilon$ in Assumption 2 and $C^\infty_0(\Omega_x)$ being dense in  $L^2(\Omega_x)$.

\smallskip
(ii) With $W_\varepsilon^0$ weakly convergent to $W_0$ in $L^2(\Omega_x)$, for $W=\rho, q$, we have
$$\lim_{\varepsilon\rightarrow0}\rho_\AP^0=\lim_{\varepsilon\rightarrow0}\pi_h\rho_\varepsilon^0=\lim_{\varepsilon\rightarrow0}\sum_{j=1}^{N_k}(\rho_\varepsilon^0,\Psi_j)\Psi_j=\sum_{j=1}^{N_k}\lim_{\varepsilon\rightarrow0}(\rho_\varepsilon^0,\Psi_j)\Psi_j=\sum_{j=1}^{N_k}(\rho_0,\Psi_j)\Psi_j=\pi_h\rho_0=\rho^0_\LIM,$$
$$\lim_{\varepsilon\rightarrow0}u_\AP^0=\lim_{\varepsilon\rightarrow0}\pi_h(\sigma_s^{-1}q_\varepsilon^0)=\lim_{\varepsilon\rightarrow0}\sum_{j=1}^{N_k}(\sigma_s^{-1}q_\varepsilon^0,\Psi_j)\Psi_j=\sum_{j=1}^{N_k}(\sigma_s^{-1}q_0,\Psi_j)\Psi_j=\pi_h(\sigma_s^{-1}q_0)=u^0_\LIM.$$

Now we consider any $ \zeta\in L^2(\Omega_v)$. With $\lgl  \zeta g_\varepsilon^0\rgl$ weakly convergent to $\lgl  \zeta g_0\rgl$ in $L^2(\Omega_x)$, we have for any $x\in\Omega_x$,
\begin{align}
\lim_{\varepsilon\rightarrow0}\lgl  \zeta, g_\AP^0(x,\cdot)\rgl
&=\lim_{\varepsilon\rightarrow0}\lgl  \zeta, \sum_{j=1}^{N_k} (g_\varepsilon^0,\Psi_j)\Psi_j(x)\rgl=\sum_{j=1}^{N_k}\lim_{\varepsilon\rightarrow0}(\lgl  \zeta g_\varepsilon^0\rgl,\Psi_j)\Psi_j(x) \notag\\
&=\sum_{j=1}^{N_k}(\lgl  \zeta g_0\rgl,\Psi_j)\Psi_j(x)
=\lgl  \zeta, g^0_\LIM(x,\cdot)\rgl.
\end{align}
And \eqref{eq:g_initial_limit2} can be proved similarly.

\smallskip
(iii) Note that 
$$ |||g_\AP^0|||^2=\lgl ||g_\AP^0||^2\rgl=\lgl\sum_{j=1}^{N_k}(g_\varepsilon^0,\Psi_j)^2\rgl\leq |||g_\varepsilon^0|||^2\sum_{j=1}^{N_k}||\Psi_j||^2=N_k|||g_\varepsilon^0|||^2,$$
$$ ||u_\AP^0||=||\pi_h(\sigma_s^{-1} q_\eps^0)||\leq
||\sigma_s^{-1} q_\eps^0||\leq 
\sigma_m^{-1}|| q_\eps^0||.$$
With Assumption 2, we have $\sup_{\varepsilon}|||W_\AP^0|||<\infty, W=g, u $. Similar proof goes to $\rho$.

\smallskip
(iv) 
Based on \eqref{eq:FDG:1T}, one has
\begin{align}
(\rho_\AP^1,\testR)&=\Dt \omega \lgl v^2 \rgl l_h(u^1_\AP,\testR)+(\rho^0_\AP,\testR)
\notag\\
&-\Dt l_h(\lgl v (g^0_\AP+\omega v u_\AP^0)\rgl,\testR)-(\sigma_a \rho_\AP^1,\testR)
,\; \forall \testR\in U_h^k.
\label{eq:b:1}
\end{align}
Take $\testR=\rho_\AP^1$, use $l_h(u^1_\AP,\rho_\AP^1)=-(\sigma_s u^1_\AP,u^1_\AP)$ based on \eqref{eq:prop:1} and Assumption 3 for $\omega$, we get when $\varepsilon<\varepsilon_0(h)$,
 \begin{align}
&|| \rho_\AP^1||^2+(\sigma_a\rho_\AP^1,\rho_\AP^1)+\frac{2\sigma_m\Dt }{3}\lgl v^2\rgl ||u_\AP^1||^2\notag\\
\leq&
(\rho^0_\AP,\rho^1_\AP)-\Dt l_h(\lgl v (g^0_\AP+\omega vu_\AP^0)\rgl,\rho^1_\AP).%\notag\\
%&=(\rho^0_\AP,\rho^1_\AP)-\Dt l_h(\lgl v g^0_\AP\rgl,\rho^1_\AP)-\lgl v^2\rgl Dt().
\label{eq:b:2}
\end{align}
Following some standard steps to  apply Cauchy-Schwarz inequality, Young inequality, inverse inequality (see, e.g. Lemma 3.9 in \cite{jang2014analysis}), based on Assumption 3, we can find a constant $C(k, \Dt, h, \Omega_v)$ such that
\begin{align}
&|(\rho^0_\AP,\rho^1_\AP)-\Dt l_h(\lgl v (g^0_\AP+\omega vu_\AP^0,\rho^1_\AP)\rgl|\notag\\
\leq &C(k, \Dt, h, \Omega_v)\left(||\rho_\AP^0||+|||g_\AP^0|||+||u_\AP^0 ||\right)||\rho_\AP^1||.
\label{eq:b:3}
\end{align}
Combining \eqref{eq:b:2}-\eqref{eq:b:3}, with $\sigma_a(x)\geq 0$, we obtain
\begin{align*}
&\sup_{0<\varepsilon<\eps_0(h)}  || \rho_\AP^1||\leq C(k, \Dt, h, \Omega_v)\sup_\varepsilon(||\rho_\AP^0||+|||g_\AP^0|||+||u_\AP^0 ||) <\infty,\\
&\sup_{0<\varepsilon<\eps_0(h)} || u_\AP^1||\leq \sqrt{\frac{3}{2\sigma_m\Dt\lgl v^2\rgl}}C(k, \Dt, h, \Omega_v)\sup_\varepsilon(||\rho_\AP^0||+|||g_\AP^0|||+||u_\AP^0 ||) <\infty.
\end{align*}
\end{proof}

%\bigskip
We are ready to prove  Theorem \ref{thm:AP} on the AP property of the IMEX1-LDG method. 
%%%%%%%%%%%%%%%%%%%%%%%%%%%
%%  proof of the theorem %%
%%%%%%%%%%%%%%%%%%%%%%%%%%%
\begin{proof}[\bf{Proof of Theorem \ref{thm:AP}.}] Let the mesh size $h$ be fixed.

{\sf Step 1:}   we first show that 
$\sup_{0<\varepsilon<\varepsilon_0(h)}||U_\AP^n||<\infty$
for any $\Dt$, $n\geq 1$, where  $W=\rho, g, q, u$. 
First note that when $\eps<\eps_0(h)$, from Assumption 3, we have $2>\omega>\frac{2}{3}$ and $\mu=\frac{3}{4}\in (\frac{1}{2\omega}, 1]$. Based on the $\mu$-stability result in Theorem \ref{thm:mu}, we have
\begin{align}
&||\rho_\AP^{n+1}||^2+\varepsilon^2|||g_\AP^{n}|||^2+\Dt\sigma_m\left( \frac{1}{4}|||g_\AP^{n}|||^2+\frac{2}{3 }\lgl v^2 \rgl ||u_\AP^{n+1}||^2\right)
\notag\\
\leq& E_{h,\mu=\frac{3}{4}}^{n+1}\leq  E_{h,\mu=\frac{3}{4}}^{n}\leq\cdots
\leq E_{h,\mu=\frac{3}{4}}^1
\notag\\
\leq& ||\rho_\AP^1||^2+\varepsilon^2|||g_\AP^0|||^2+\Dt\sigma_M \left( \frac{1}{4}|||g_\AP^{0}|||^2+2\lgl v^2 \rgl ||u_\AP^1||^2\right).
\label{eq:prove-AP} \end{align}
Moreover from \eqref{eq:FDG:1T:c}, we have $||q_\AP^n||^2=(\sigma_s u_\AP^n,q_\AP^n)$, hence  $||q_\AP^n||\leq \sigma_M ||u_\AP^n||$. In combination of Lemma \ref{lem:AP},  the finiteness of $\sup_{0<\varepsilon<\varepsilon_0(h)}||W_\AP^n||$, $\forall n\geq 1$ follows for  $W=\rho, g, q, u$.

\smallskip
{\sf Step 2:}
With Lemma \ref{lem:AP}, we only need to establish \eqref{eq:f:lim}  for any  $n\geq 1$. This is equivalent to show that for any given sequence $\{\varepsilon_m\}_{m=1}^{\infty}$, satisfying  $\lim_{m\rightarrow \infty} \varepsilon_m=0$ (we no longer emphasize that $\varepsilon$ considered here is bounded above by $\varepsilon_0(h)$),  we have
\begin{subequations}
\label{eq:aa:1}
\begin{align}
\label{eq:aa:1-1}
& \lim_{m\rightarrow \infty}W_{\varepsilon_m, \Dt, h}^n  =W_{\Dt, h}^n, \quad W=\rho, q, u,\\
& \lim_{m\rightarrow \infty}\lgl \zeta, g_{\varepsilon_m, \Dt, h}^n(x,\cdot)\rgl =\lgl \zeta, g_\LIM^n(x, \cdot)\rgl,\quad \forall \zeta\in L^2(\Omega_v),\;\;\forall x\in \Omega_x,
\label{eq:aa:1-2}\\
& \lim_{m\rightarrow \infty}\lgl \zeta, (g_{\varepsilon_m, \Dt, h}^n, \psi) \rgl =\lgl \zeta, (g_\LIM^n, \psi)\rgl,\quad \forall \zeta\in L^2(\Omega_v), \;\;\forall \psi\in L^2(\Ox),
\label{eq:aa:1-3}
\end{align}
\end{subequations}
 for some $W_\LIM^n\in U_h^k$, with $W=\rho, q, u$,  and $g_\LIM^n\in G_h^k$, $\forall n\geq 1$.
   Let $W$ be any of $\rho, q, u$. Given that $U_h^k$ is finite dimensional,  the finiteness of $\sup_m ||W_{\varepsilon_m, \Dt, h}^n||$ from  {\sf Step 1} implies that there is a subsequence $\{W_{\varepsilon_{m_r}, \Dt, h}^n\}_{r=1}^{\infty}$ converging in $U_h^k$ under {\em any} norm as $r\rightarrow \infty$.  Let the limit be 
   \beq
   \label{eq:aa:2}
    W_{\Dt, h}^n=\lim_{r\rightarrow \infty}W_{\varepsilon_{m_r}, \Dt, h}^n, \quad W=\rho, q, u.
    \eeq

We now turn to $\{g_{\varepsilon_m, \Dt, h}^n\}_{m=1}^\infty$. 
Note that each $g^n_{\varepsilon_m, \Dt, h}$ can be written as $g^n_{\varepsilon_m, \Dt, h}(x,v)=\sum_{j=1}^{N_k}\alpha_{\varepsilon_m}^{(j)}(v)\Psi_j(x)$, with
$||| g^n_{\varepsilon_m}|||=\left(\sum_{j=1}^{N_k} ||\alpha_{\varepsilon_m}^{(j)}||^2_{L^2(\Omega_v)}\right)^{1/2}$. This, in addition to the finiteness of
$\sup_m |||g_{\varepsilon_m, \Dt, h}^n|||$ in {\sf Step 1},  indicates that $\sup_r ||\alpha_{\varepsilon_{m_r}}^{(j)}||^2_{L^2(\Omega_v)}$  is bounded for any $j=1,\cdots, N_k$. As a Hilbert space, $L^2(\Omega_v)$ is weakly sequentially compact, that is, $\{\alpha_{\varepsilon_{m_r}}^{(j)}\}_{r=1}^\infty$ has a subsequence which is weakly convergent in $L^2(\Omega_v)$. Without loss of generality, this subsequence is still denoted as $\{\alpha_{\varepsilon_{m_r}}^{(j)}\}_{r=1}^\infty$, and the weak limit when $r\rightarrow \infty$ is denoted as $\alpha_0^{(j)}\in L^2(\Omega_v)$, $\forall j$. We now define
$g_{\Dt, h}^n(x,v)=\sum_{j=1}^{N_k}\alpha_0^{(j)}(v)\Psi_j(x)$.
It is clear that
{
$g_{\Dt, h}^n\in G_h^k$.
}
For any $\zeta\in L^2(\Omega_v)$, and any $x\in\Omega_x$,
\begin{equation}
\label{eq:aa:3}
\lim_{r\rightarrow \infty}\lgl \zeta, g_{\varepsilon_{m_r}, \Dt, h}^n(x,\cdot)\rgl=\sum_{j=1}^{N_k} \left(\lim_{r\rightarrow\infty}\langle \zeta, \alpha_{\varepsilon_{m_r}}^{(j)}\rangle \right)\Psi_j(x)=\sum_{j=1}^{N_k} \langle \zeta, \alpha_0^{(j)}\rangle \Psi_j(x)=\lgl \zeta, g_{\Dt, h}^n(x,
\cdot)\rgl.
\end{equation}
Furthermore, we have $\forall \zeta\in L^2(\Omega_v),\;\forall \psi\in L^2(\Ox),$
\begin{equation}
\label{eq:aa:4}
\lim_{r\rightarrow \infty}\lgl \zeta, (g_{\varepsilon_{m_r}, \Dt, h}^n, \psi)\rgl
=\sum_{j=1}^{N_k} \left(\lim_{r\rightarrow\infty}\langle \zeta, \alpha_{\varepsilon_{m_r}}^{(j)}\rangle \right)(\Psi_j, \psi)
=\lgl \zeta, (g_{\Dt, h}^n, \psi)\rgl=(\lgl \zeta g_{\Dt, h}^n\rgl, \psi).
\end{equation}

Use \eqref{eq:aa:2}-\eqref{eq:aa:4} for $n\geq 1$ as well as the similar result in Lemma \ref{lem:AP} for $n=0$, with   $\zeta$ taken when needed as
$v$,  $v {\bf 1}_{\{v>0\}}$,  $v{\bf 1}_{\{v<0\}}$, $v\zeta(v)$, $v\zeta(v){\bf 1}_{\{v>0\}}$,  $v\zeta(v){\bf 1}_{\{v<0\}}$, also use the property \eqref{w:prop:0} for $\omega$, 
 we have for any $n\geq 0$,
\begin{subequations}
\label{eq:aa:5}
\begin{align}
&
 \lim_{r\rightarrow \infty} l_h(\lgl v (g_{\varepsilon_{m_r}, \Dt, h}^n+ \omega|_{\eps=\eps_{m_r}}v u_{\varepsilon_{m_r},\Dt, h}^n\rgl ,\testR)=
 l_h(\lgl v (g_\LIM^n+v u_{\LIM}^n)\rgl ,\testR) ,\quad \forall \testR\in U_h^k, \\
&
\lim_{r\rightarrow \infty} \lgl \zeta, b_{h,v}(g_{\varepsilon_{m_r}, \Dt, h}^n,\testG)\rgl=\lgl \zeta, b_{h,v}(g_\LIM^n,\testG)\rgl,\quad \forall \zeta\in L^2(\Omega_v), \; \forall \testG\in U_h^k.
\end{align}
\end{subequations}

Now with \eqref{eq:aa:2}-\eqref{eq:aa:5} and  Lemma \ref{lem:AP} for the initial data, the numerical scheme  \eqref{eq:FDG:1T} as  $r\rightarrow \infty$  becomes,
$\forall \testQ,\;\testU,\;\testR\,\;\testG \in U_h^k$
\begin{subequations}
\label{eq:aa:6}
\begin{align}
&(q_\LIM^{n+1},\testQ)+d_h(\rho_\LIM^{n+1},\testQ)=0,
\label{eq:aa:6.1}\\
&(\sigma_s u_\LIM^{n+1},\testU)=(q_\LIM^{n+1},\testU),
\label{eq:aa:6.4}\\
&(\frac{\rho_\LIM^{n+1}-\rho_\LIM^n}{\Dt},\testR)+l_h(\lgl v(g_\LIM^n+v u_\LIM^n)\rgl,\testR)=\lgl v^2\rgl l_h(u_\LIM^{n+1},\testR)-(\sigma_a \rho_\LIM^{n+1},\testR),
\label{eq:aa:6.2}\\
&(\lgl \zeta\sigma_s g_\LIM^{n+1}\rgl ,\testG)=\lgl \zeta v\rgl d_h(\rho_\LIM^{n+1},\testG),\quad\forall \zeta\in L^2(\Omega_v),
\label{eq:aa:6.3}
\end{align}
\end{subequations}
for $n\geq 0$.
Furthermore, \eqref{eq:aa:6.1} and \eqref{eq:aa:6.3} lead to
\beq 
\lgl  (\pi_h(\sigma_s g_\LIM^n)+vq_\LIM^n, \zeta\psi) \rgl=0
\quad \forall \zeta\in L^2(\Omega_v),
 \;\psi \in U_h^k, \;\; n\geq 1.
\label{eq:aa:7}
\eeq
With $g_\LIM^{n}\in G_h^k$ hence $\pi_h(\sigma_s g_\LIM^n)+vq_\LIM^n\in L^2(\Ov)\times U_h^k$, \eqref{eq:aa:7} equivalently becomes
\beq
\pi_h(\sigma_s g_\LIM^{n})=- v q_\LIM^{n},\quad n\geq 1.
\label{eq:aa:8}
\eeq
Moreover, from \eqref{eq:aa:6.4} and \eqref{eq:aa:8}, one can get $g_\LIM^n+vu^n_\LIM=0, n\geq 1$, as shown  below.
\begin{align*}
0\leq \sigma_m ||| g_\LIM^n+vu^n_\LIM |||^2
&\leq \left\lgl \left(\sigma_s (g_\LIM^n+vu^n_\LIM), g_\LIM^n+vu^n_\LIM\right)\right\rgl\\
&=\left\lgl\left( -vq_\LIM^n+vq_\LIM^n,g_\LIM^n+vu^n_\LIM\right)\right\rgl=0.
\end{align*}

Compare \eqref{eq:aa:6} and \eqref{eq:aa:8} with what we want in \eqref{eq:lim:s}, one also needs to have 
$g_\LIM^0+vu^0_\LIM=0$. This can be argued based on the initial data being  well-prepared in Assumption 1. To see this, $\forall \zeta\in L^2(\Ov), \forall \psi\in U_h^k$, we proceed as follows,
\begin{align}
0&=\lim_{\varepsilon\rightarrow 0}\Big(\lgl \zeta ( g_\varepsilon^{0}+v \sigma_s^{-1}q_\varepsilon^{0})\rgl,  \psi\Big)=\lim_{\varepsilon\rightarrow 0}\Big((\lgl \zeta g_\varepsilon^{0}\rgl, \psi)+\lgl v\zeta\rgl (q_\varepsilon^{0}, \sigma_s^{-1}\psi)\Big)\notag \\
&=(\lgl \zeta g_0\rgl, \psi)+\lgl v\zeta\rgl (q_0, \sigma_s^{-1}\psi)
=(\lgl  \zeta  g_\LIM^{0}\rgl, \psi)+ \lgl \zeta v\rgl ( u_\LIM^{0}, \psi),
\label{eq:aa:7.a}
\end{align}
and this gives $\lgl \zeta (g_\LIM^{0}+vu_\LIM^{0}, \psi)\rgl=0$. Note that $g_\LIM^{0}+vu_\LIM^{0}\in L^2(\Ov)\times U_h^k$, therefore \eqref{eq:aa:7.a} is indeed $g_\LIM^{0}+vu_\LIM^{0}=0$, and we can conclude the limiting scheme in   \eqref{eq:lim:s}.

It is easy to see the limiting scheme \eqref{eq:lim:s} is a consistent discretization for \eqref{eq:AP1:3.1.1}. Its stability can be obtained similarly as Lemma \ref{lem:sol}, with 
\begin{align}
&||\rho^{n+1}_\LIM||^2+\Dt\lgl v^2\rgl ||u^{n+1}_\LIM||_s^2+(\sigma_a \rho^{n+1}_\LIM, \rho^{n+1}_\LIM)=(\rho^{n}_\LIM, \rho^{n+1}_\LIM)\notag\\
\Rightarrow &\frac{1}{2}||\rho^{n+1}_\LIM||^2+\Dt\lgl v^2\rgl \sigma_m ||u^{n+1}_\LIM||^2\leq \frac{1}{2}||\rho^{n}_\LIM||^2\leq \cdots\leq \frac{1}{2}||\rho^{0}_\LIM||^2\leq \frac{1}{2}||\rho_0||^2. 
\end{align}

Finally, with a standard contradiction argument and the uniqueness of the solution to the system
\eqref{eq:lim:s} (see Lemma \ref{lem:sol}), we conclude the limiting functions $\rho^n_\LIM,  q^n_\LIM, g^n_\LIM, u_\LIM^n$ are unique, and  \eqref{eq:aa:1} holds for the entire sequence. In the case that the velocity space $\Omega_v$ is discrete, the analysis related to the convergence of $g_\AP^n(\cdot, v)$ for each $v$ is just as simple as that for $\rho_\AP^n$ and $q_\AP^n$, and the convergence is in a strong sense as in \eqref{eq:aa:-4}.

\end{proof}

\renewcommand\refname{References}
\bibliographystyle{plain}
\bibliography{bib_peng}

\begin{thebibliography}{10}

\bibitem{boscarino2014high}
Sebastiano Boscarino, Philippe~G LeFloch, and Giovanni Russo.
\newblock High-order asymptotic-preserving methods for fully nonlinear
  relaxation problems.
\newblock {\em SIAM Journal on Scientific Computing}, 36(2):A377--A395, 2014.

\bibitem{boscarino2013imex}
Sebastiano Boscarino, Lorenzo Pareschi, and Giovanni Russo.
\newblock Implicit-explicit {R}unge-{K}utta schemes for hyperbolic systems and
  kinetic equations in the diffusion limit.
\newblock {\em SIAM Journal on Scientific Computing}, 35(1):A22--A51, 2013.

\bibitem{caflisch1997uniformly}
Russel~E Caflisch, Shi Jin, and Giovanni Russo.
\newblock Uniformly accurate schemes for hyperbolic systems with relaxation.
\newblock {\em SIAM Journal on Numerical Analysis}, 34(1):246--281, 1997.

\bibitem{cockburn1998local}
Bernardo Cockburn and Chi-Wang Shu.
\newblock The local discontinuous {G}alerkin method for time-dependent
  convection-diffusion systems.
\newblock {\em SIAM Journal on Numerical Analysis}, 35(6):2440--2463, 1998.

\bibitem{degond2011asymptotic}
Pierre Degond.
\newblock Asymptotic-preserving schemes for fluid models of plasmas.
\newblock {\em arXiv preprint arXiv:1104.1869}, 2011.

\bibitem{dimarco2014implicit}
Giacomo Dimarco, Lorenzo Pareschi, and Vittorio Rispoli.
\newblock Implicit-explicit {R}unge-{K}utta schemes for the
  {B}oltzmann-{P}oisson system for semiconductors.
\newblock {\em Communications in Computational Physics}, 15(5):1291--1319,
  2014.

\bibitem{filbet2013analysis}
Francis Filbet and Am{\'e}lie Rambaud.
\newblock Analysis of an asymptotic preserving scheme for relaxation systems.
\newblock {\em ESAIM: Mathematical Modelling and Numerical Analysis},
  47(2):609--633, 2013.

\bibitem{golse1999convergence}
Fran{\c{c}}ois Golse, Shi Jin, and C~David Levermore.
\newblock The convergence of numerical transfer schemes in diffusive regimes i:
  Discrete-ordinate method.
\newblock {\em SIAM journal on numerical analysis}, 36(5):1333--1369, 1999.

\bibitem{hesthaven2007nodal}
Jan~S Hesthaven and Tim Warburton.
\newblock {\em Nodal discontinuous Galerkin methods: algorithms, analysis, and
  applications}.
\newblock Springer Science \& Business Media, 2007.

\bibitem{hu2019uniform}
Jingwei Hu and Ruiwen Shu.
\newblock On the uniform accuracy of implicit-explicit backward differentiation
  formulas (imex-bdf) for stiff hyperbolic relaxation systems and kinetic
  equations.
\newblock {\em arXiv preprint arXiv:1912.00559}, 2019.

\bibitem{jang2014analysis}
Juhi Jang, Fengyan Li, Jing-Mei Qiu, and Tao Xiong.
\newblock Analysis of asymptotic preserving {DG}-{IMEX} schemes for linear
  kinetic transport equations in a diffusive scaling.
\newblock {\em SIAM Journal on Numerical Analysis}, 52(4):2048--2072, 2014.

\bibitem{AP1}
Juhi Jang, Fengyan Li, Jing-Mei Qiu, and Tao Xiong.
\newblock High order asymptotic preserving {DG}-{IMEX} schemes for
  discrete-velocity kinetic equations in a diffusive scaling.
\newblock {\em Journal of Computational Physics}, 281:199--224, 2015.

\bibitem{jin2010asymptotic}
Shi Jin.
\newblock Asymptotic preserving ({AP}) schemes for multiscale kinetic and
  hyperbolic equations: a review.
\newblock {\em Lecture Notes for Summer School on Methods and Models of Kinetic
  Theory (M\&MKT), Porto Ercole (Grosseto, Italy)}, 2010.

\bibitem{jin1998diffusive}
Shi Jin, Lorenzo Pareschi, and Giuseppe Toscani.
\newblock Diffusive relaxation schemes for multiscale discrete-velocity kinetic
  equations.
\newblock {\em SIAM Journal on Numerical Analysis}, 35(6):2405--2439, 1998.

\bibitem{klar1998asymptotic}
Axel Klar.
\newblock An asymptotic-induced scheme for nonstationary transport equations in
  the diffusive limit.
\newblock {\em SIAM journal on numerical analysis}, 35(3):1073--1094, 1998.

\bibitem{klar2002uniform}
Axel Klar and Andreas Unterreiter.
\newblock Uniform stability of a finite difference scheme for transport
  equations in diffusive regimes.
\newblock {\em SIAM Journal on Numerical Analysis}, 40(3):891--913, 2002.

\bibitem{Lemou_No_BC}
Mohammed Lemou and Luc Mieussens.
\newblock A new asymptotic preserving scheme based on micro-macro formulation
  for linear kinetic equations in the diffusion limit.
\newblock {\em SIAM Journal on Scientific Computing}, 31(1):334--368, 2008.

\bibitem{liu2010analysis}
Jian-Guo Liu and Luc Mieussens.
\newblock Analysis of an asymptotic preserving scheme for linear kinetic
  equations in the diffusion limit.
\newblock {\em SIAM Journal on Numerical Analysis}, 48(4):1474--1491, 2010.

\bibitem{liu2004boltzmann}
Tai-Ping Liu and Shih-Hsien Yu.
\newblock Boltzmann equation: micro-macro decompositions and positivity of
  shock profiles.
\newblock {\em Communications in mathematical physics}, 246(1):133--179, 2004.

\bibitem{naldi1998numerical}
Giovanni Naldi and Lorenzo Pareschi.
\newblock Numerical schemes for kinetic equations in diffusive regimes.
\newblock {\em Applied mathematics letters}, 11(2):29--35, 1998.

\bibitem{peng2020thesis}
Zhichao Peng.
\newblock {\em Structure-preserving discontinuous Galerkin methods for
  multi-scale kinetic transport equations and nonlinear optics models}.
\newblock PhD thesis, Rensselaer Polytechnic Institute, 2020.

\bibitem{peng2018stability}
Zhichao Peng, Yingda Cheng, Jing-Mei Qiu, and Fengyan Li.
\newblock Stability-enhanced {AP} {IMEX}-{LDG} schemes for linear kinetic
  transport equations under a diffusive scaling.
\newblock 2018.

\bibitem{pomraning1973equations}
Gerald~C. Pomraning.
\newblock {The equations of radiation hydrodynamics}.
\newblock {\em International Series of Monographs in Natural Philosophy,
  Oxford: Pergamon Press}, 1973.

\bibitem{Matt_Li}
Matthew~A Reyna and Fengyan Li.
\newblock Operator bounds and time step conditions for the {DG} and central
  {DG} methods.
\newblock {\em Journal of Scientific Computing}, 62(2):532--554, 2015.

\end{thebibliography}
\end{document}